\documentclass[11pt]{article}
\usepackage[a4paper,margin=1in]{geometry}
\usepackage{amsmath,amssymb,amsthm,mathtools,bm,mathrsfs}
\usepackage{enumitem}
\usepackage{hyperref}
\usepackage{comment}
\usepackage{microtype}
\usepackage{titlesec}
\usepackage{fancyhdr}
\usepackage{setspace}
\usepackage{lmodern}
\hypersetup{colorlinks=true,linkcolor=blue,citecolor=blue,urlcolor=blue}
\titleformat{\section}{\large\bfseries}{\thesection.}{0.6em}{}
\titleformat{\subsection}{\normalsize\bfseries}{\thesubsection.}{0.5em}{}

\newtheorem{theorem}{Theorem}[section]
\newtheorem{proposition}[theorem]{Proposition}
\newtheorem{lemma}[theorem]{Lemma}
\newtheorem{corollary}[theorem]{Corollary}
\theoremstyle{definition}
\newtheorem{definition}[theorem]{Definition}
\newtheorem{remark}[theorem]{Remark}

\newcommand{\R}{\mathbb{R}}
\newcommand{\C}{\mathbb{C}}
\newcommand{\Z}{\mathbb{Z}}

\newcommand{\Fa}{\mathcal{F}_{\alpha}}
\newcommand{\Ma}{M_{\alpha}}
\newcommand{\Mai}{M_{\alpha}^{-1}}
\newcommand{\sa}{s_{\alpha}}
\newcommand{\ka}{\kappa_{\alpha}}
\newcommand{\BMOa}{\mathrm{BMO}_{\alpha}}
\newcommand{\Hpa}{H^{p}_{\alpha}}
\newcommand{\Avg}{\mathrm{Avg}}
\newcommand{\supp}{\mathrm{supp}}
\newcommand{\norm}[1]{\left\lVert #1 \right\rVert}
\newcommand{\abs}[1]{\left| #1 \right|}

\newcommand{\calS}{\mathcal{S}}

\newcommand{\Tem}{\mathcal{S}'}
\newcommand{\Da}{\mathfrak{D}(\alpha)}

\newcommand{\br}[1]{\left\langle #1 \right\rangle}

\newcommand{\set}[1]{\left\{ #1 \right\}}

\newcommand{\calP}{\mathcal P}
\newcommand{\calM}{\mathcal M}
\newcommand{\calC}{\mathcal C}

\newcommand{\BMO}{\mathrm{BMO}}
\newcommand{\Hp}{H^p}
\newcommand{\Msharp}{M^{\#}}

\newcommand{\Palpha}{\mathcal P_{\alpha}}
\newcommand{\staralpha}{\star_{\alpha}}

\newcommand{\paren}[1]{\left( #1 \right)}

\title{Recent progress of Littlewood-paley Theory with chirp function}
\date{} % 设置日期为空
\author{Xiang Li Qianjun He and Zunwei Fu}

\begin{document}
\maketitle
\begin{abstract}
Littlewood--Paley theory is a fundamental tool for frequency localization, square-function control, and multiplier analysis, yet a systematic counterpart in the fractional Fourier transform (FrFT) setting has remained incomplete. We develop a unified FrFT Littlewood--Paley framework based on the observation that, for a fixed $\alpha\notin\pi\mathbb Z$, a broad class of FrFT-side operators are exact chirp conjugates of their classical Fourier counterparts through
$$
M_{\alpha}f(x)=e^{i\pi |x|^2\cot\alpha}f(x).
$$

Within this unified framework we present: the FrFT multiplier identity; Littlewood--Paley square-function estimates and the converse theorem; sharp dyadic interval decompositions; Marcinkiewicz and Mihlin--H"ormander multiplier results; maximal, rough square-function, and almost-orthogonality estimates; twisted dyadic martingale geometry; inhomogeneous Sobolev, Besov, and Triebel--Lizorkin descriptions; Calder\'on reproducing formulae; pullback spaces and FrFT Riesz--Bessel operators; BMO, Carleson, sharp-maximal, and Hardy-space; twisted product estimates, multilinear bounds, and a Kato--Ponce theorem; fractional order-shifting in Lipschitz spaces; and the classical limit and singular boundary laws for the fractional parameter. The recurring theme is that a large class of FrFT operators are exact chirp conjugates of their classical counterparts, so most estimates are inherited with the same constants after one time identification of the rescaled symbols.
\end{abstract}

\section{Introduction}
\quad

In harmonic analysis, Littlewood–Paley theory constitutes a cornerstone framework that converts precise frequency localization into robust square function estimates. This mechanism furnishes a singular analytic architecture unifying a wide array of central topics, including Fourier multiplier operators, maximal operators, vector-valued inequalities, rough square functions, and dyadic decompositions. Originating from the pioneering work of Littlewood and Paley on dyadic decompositions of Fourier series \cite{Littlewood1931Theorems}, the theory has rapidly evolved from one-dimensional roots to a high-dimensional methodology quantifying the almost orthogonality inherent in frequency bands. At the heart of this machinery lies the Littlewood–Paley g-function and its variants, whose $L^p$ boundedness—originally established via intricate real-variable arguments \cite{Stein1970Topics}—was later elegantly recast as a direct consequence of the vector-valued Calder\'{o}n–Zygmund theory \cite{Stein1993Harmonic}. The ability to handle vector-valued extensions and dimension-free estimates has been central to the theory’s development, with landmark contributions by Bourgain \cite{Bourgain1986Vector} and Rubio de Francia \cite{Rubio1985Dimension} deepening our understanding of such inequalities. Meanwhile, the underlying dyadic structure remains indispensable for analyzing singular integrals with rough kernels and for the systematic study of maximal functions \cite{Edwards1977Littlewood}.

The enduring vitality of Littlewood–Paley theory is further evidenced by its far-reaching applications beyond classical singular integrals: it has become an essential tool in the analysis of partial differential equations, particularly through frequency-localization techniques for nonlocal problems and evolution equations, and its frontiers continue to advance in multilinear settings and on noncommutative spaces \cite{Frazier1991Littlewood,Stein1993Harmonic}. For comprehensive treatments, we refer the reader to the classical monographs and surveys that have shaped the modern theory \cite{Stein1970Topics,Zygmund2002Trigonometric}.

A large class of FrFT operators are chirp-conjugated versions of their classical Fourier counterparts. For a fixed angle $\alpha\in\R\setminus \pi\mathbb Z$, define
$$\Ma f(x):=e^{i\pi |x|^2\cot\alpha}f(x),\qquad\sa:=|\sin\alpha|>0.$$

Since $|\Ma f|=|f|$ pointwise almost everywhere, the operator $\Ma$ is an isometry on every $L^p$ space. This elementary fact underpins the exact transfer of square functions, vector-valued norms, weak-type level sets, and maximal expressions without any loss of constants.

Throughout the paper, $C$ denotes a positive constant independent of the main parameters under consideration, whose value may change from line to line. The notation $A\lesssim B$ means $A\le CB$, and $A\approx B$ means both $A\lesssim B$ and $B\lesssim A$. Unless otherwise stated, all $L^p$ norms are taken over $\R^n$.

\section{Preliminaries and notations}

Throughout this section, let $n\ge 1$ and $\alpha\in \R\setminus \pi\Z$ be fixed. We write
$$
\sa:=|\sin\alpha|>0.
$$

\begin{definition}[\textbf{Fourier transform and chirp multiplier}]\label{lemma-2.1}
For $f\in \mathcal S(\R^n)$, define the Fourier transform and the inverse Fourier transform by
$$
\widehat f(\xi)=\int_{\R^n} e^{-2\pi i x\cdot \xi}f(x)\,dx,
\qquad
f^\vee(x)=\int_{\R^n} e^{2\pi i x\cdot \xi}\widehat f(\xi)\,d\xi.
$$
Define the chirp multiplier $\Ma$ by
$$
(\Ma f)(x):=e^{\pi i |x|^2\cot\alpha}f(x).
$$
\end{definition}

Since $\bigl|e^{\pi i |x|^2\cot\alpha}\bigr|=1$, the operator $\Ma$ is an isometry on $L^p(\R^n)$ for every $0<p\le \infty$, i.e.,
$$
\|\Ma f\|_{L^p}=\|f\|_{L^p}.
$$

\begin{definition}[\textbf{Fractional Fourier transform}]
The fractional Fourier transform at angle $\alpha$ is defined by
\begin{equation}\label{eq:frft-def}
(\Fa f)(\xi)
=
c_{\alpha,n}\,\sa^{-n/2}\,
e^{\pi i |\xi|^2\cot\alpha}
\widehat{\Ma f}\!\left(\frac{\xi}{\sin\alpha}\right),
\end{equation}
where $c_{\alpha,n}$ is a unimodular constant depending only on $(\alpha,n)$.
\end{definition}

The precise value of $c_{\alpha,n}$ is irrelevant for multiplier identities and norm comparisons below.

\begin{definition}[\textbf{Fourier and FrFT multipliers}]\label{definition-2.3}
Let $m$ be a bounded measurable function on $\R^n$. The classical Fourier multiplier with symbol $m$ is defined by
$$
T_m g := (m\,\widehat g)^\vee.
$$
The corresponding FrFT multiplier is defined by
$$
T_{m,\alpha}f:=\Fa^{-1}\bigl(m\,\Fa f\bigr).
$$
\end{definition}

\begin{definition}[\textbf{Localized FrFT spectral family and FrFT Bochner--Riesz means}]
Let $\Phi\in L^\infty(\R^n)$ and $R>0$. Define
$$
S^\Phi_{R,\alpha}f:=\Fa^{-1}\!\bigl(\Phi(\cdot/R)\,\Fa f\bigr).
$$
For $\lambda\ge 0$, define the FrFT Bochner--Riesz means by
$$
B^\lambda_{R,\alpha}f
:=
\Fa^{-1}\!\left(
\left(1-\frac{|\xi|^2}{R^2}\right)_+^\lambda
\Fa f(\xi)
\right).
$$
\end{definition}

\begin{definition}[\textbf{Littlewood--Paley operator}]
Choose $\varphi\in C_c^\infty(\R^n\setminus\{0\})$ supported in
$$
\left\{\xi\in\R^n:\frac12\le |\xi|\le 2\right\}
$$
and satisfying
$$
\sum_{j\in\Z}\varphi(2^{-j}\xi)=1
\qquad (\xi\neq 0).
$$
The homogeneous dyadic pieces are defined by
$$
\Delta_j f := \bigl(\varphi(2^{-j}\cdot)\widehat f\bigr)^\vee.
$$
Choose also $\chi\in C_c^\infty(\R^n)$ supported in $\{|\xi|\le 2\}$ and equal to $1$ on $\{|\xi|\le 1\}$. Define the inhomogeneous low-frequency cutoff by
$$
S_0 f := (\chi\,\widehat f)^\vee.
$$
\end{definition}

For $\gamma>0$, we use the classical Littlewood--Paley characterizations
\begin{align}
\|f\|_{\dot\Lambda^\gamma}
&\asymp
\sup_{j\in\Z} 2^{j\gamma}\|\Delta_j f\|_{L^\infty},
\label{eq:hom-lip}\\
\|f\|_{\Lambda^\gamma}
&\asymp
\|S_0f\|_{L^\infty}
+
\sup_{j\ge 0} 2^{j\gamma}\|\Delta_j f\|_{L^\infty}.
\label{eq:inh-lip}
\end{align}

\begin{definition}[\textbf{Fractional derivatives, Bessel potentials, and twisted operators}]
For $\sigma>0$, define
$$
|D|^\sigma f := \bigl(|\xi|^\sigma \widehat f(\xi)\bigr)^\vee,
\qquad
\br{D}^\sigma f := \bigl((1+|\xi|^2)^{\sigma/2}\widehat f(\xi)\bigr)^\vee.
$$
For later use, define the twisted product
$$
f\star_\alpha g := \Ma^{-1}\bigl((\Ma f)(\Ma g)\bigr)
$$
and the conjugated fractional derivative
$$
\Da^s := \Ma^{-1} D^s \Ma,
$$
where $D^s$ denotes the classical Fourier multiplier with symbol $|\xi|^s$.
\end{definition}

\begin{definition}[\textbf{FrFT BMO spaces, operators, and auxiliary objects}]
\leavevmode
\begin{enumerate}[label=\textup{(\roman*)}]
\item The FrFT BMO space is
$$
\BMOa(\R^n):=\Mai(\BMO(\R^n))
=\set{b\in \calS'(\R^n):\Ma b\in \BMO(\R^n)},
$$
with norm $\norm{b}_{\BMOa}:=\norm{\Ma b}_{\BMO}$.

\item The chirped-constant class is
$$
\calC_{\alpha}:=\Mai \C=
\set{c\,e^{-i\pi |x|^2\cot\alpha}:c\in\C}.
$$
Thus $\BMOa$ is naturally considered modulo $\calC_{\alpha}$.

\item For $0<p\le 1$, define
$$
\Hpa(\R^n):=\Mai(\Hp(\R^n)),
\qquad
\norm{f}_{\Hpa}:=\norm{\Ma f}_{\Hp}.
$$

\item Let $\calP$ be the space of polynomials on $\R^n$. The \textbf{chirped-polynomial class} is
$$
\Palpha:=\Mai\calP
=\set{e^{-i\pi |x|^2\cot\alpha}P(x): P\in\calP}.
$$

\item The FrFT Littlewood-Paley operators are
$$
\Delta_{j,\alpha}:=\Mai\Delta_j\Ma,
\qquad
S_{\alpha}(f):=\paren{\sum_{j\in\Z}\abs{\Delta_{j,\alpha}f}^2}^{1/2}.
$$

\item The FrFT sharp maximal operator is
$$
\Msharp_{\alpha}f:=\Msharp(\Ma f).
$$

\item Given a Schwartz function $\Psi$ with $\int_{\R^n}\Psi=0$, write $\Psi_t(x)=t^{-n}\Psi(x/t)$ and define
$$
\Psi^{\alpha}_t b:=\Mai\bigl(\Psi_t*(\Ma b)\bigr).
$$
The associated measure on the upper half-space is
$$
 d\mu_b^{\alpha}(x,t):=\abs{\Psi_t^{\alpha}b(x)}^2\,dx\,\frac{dt}{t}.
$$

\item For a real order $s$, define the transported derivative
$$
D^s_{\alpha}:=\Mai D^s \Ma,
$$
where $D^s$ is the classical homogeneous derivative $\mathcal F^{-1}\bigl(|\xi|^s\widehat{\cdot}\,\bigr)$.

\item The \textbf{chirpconvolution} is
$$
 f\staralpha g:=\Mai\bigl((\Ma f)*(\Ma g)\bigr).
$$
\end{enumerate}
\end{definition}

\begin{definition}[\textbf{Fractional descriptors and effective bandwidth}]
For $\alpha\in\R\setminus\pi\Z$, define
$$
\ka:=\cot\alpha,
\qquad
\sa:=|\sin\alpha|,
\qquad
D(\alpha):=|1-\sa|+|\ka|.
$$
For a bounded localizing template $\Phi$ and radius $R>0$, define the classical Fourier multiplier
$$
T_{\Phi(\cdot/R)}g:=\big(\Phi(\cdot/R)\,\widehat g\big)^{\vee}.
$$
The effective classical radius is
$$
R_{\mathrm{eff}}(\alpha;R):=\frac{R}{\sa}.
$$
\end{definition}

\section{Some Lemmas}\label{sec:conj}

\begin{lemma}[\textbf{Isometric of the chirp multiplier}]\label{lem:isometry}
For every $1\le p\le \infty$ and every scalar, sequence, or Hilbert-valued function $f$, the chirp multiplier $\Ma$ acts isometrically:
$$\norm{\Ma f}_{L^p}=\norm{f}_{L^p},$$
$\Ma$ is unitary and $\Ma^{-1}=\Ma^\ast$ on $L^2(\R^n)$.
\end{lemma}

\begin{proof}
The pointwise modulus of the chirp factor is one:
$$\left|e^{i\pi |x|^2\cot\alpha}\right|=1.$$

Hence $|\Ma f(x)|=|f(x)|$ almost everywhere. Taking the $L^p$ norm yields the claim for scalar functions, and the same pointwise argument applies componentwise to sequence-valued or Hilbert-valued functions. In $L^2$, the inverse is multiplication by the complex conjugate phase, we have that $\Ma^{-1}=\Ma^\ast$ and $\Ma$ is unitary.
\end{proof}

\begin{lemma}[\textbf{FrFT factorization through chirp modulation}]\label{lem:factor}
For every Schwartz function $f$,
$$
(\Fa f)(u)=A_{\alpha,n}e^{i\pi |u|^{2}\cot\alpha}
\widehat{\Ma f}\!\left(\frac{u}{\sin\alpha}\right),
$$
where the constant $A_{\alpha,n}$ denotes $c_{\alpha,n}\sa^{-n/2}$.
Using the Definition $\ref{lemma-2.1}$, $\Fa f)(u)$ can be written
$$
(\Fa f)(u)=A_{\alpha,n}e^{i\pi |u|^{2}\cot\alpha}
\int_{\R^{n}}e^{-2\pi i x\cdot u/\sin\alpha}e^{i\pi |x|^{2}\cot\alpha}f(x)\,dx.
$$
\end{lemma}

\begin{proof}
This is the standard kernel formula for the fractional Fourier transform after factoring out the $u$-dependent quadratic phase. The remaining oscillatory integral is precisely the Fourier transform of $\Ma f$ at $u/\sin\alpha$.
\end{proof}

\begin{lemma}[\textbf{Elementary properties of the chirp multiplier}]\label{lem:isom}
For every measurable $f$, every $0<p\leq \infty$, and every cube $Q\subset\R^{n}$, the following three equations hold:
$$
\norm{\Ma f}_{L^{p}}=\norm{f}_{L^{p}},
\qquad
\supp(\Ma f)=\supp(f),
\qquad
\Avg_{Q}|\Ma f|=\Avg_{Q}|f|.
$$
The same pointwise-modulus preservation holds for sequence-valued and Hilbert-valued functions, $\Ma$ is unitary and $\Mai=\Ma^{*}$ on $L^{2}(\R^{n})$.
\end{lemma}

\begin{proof}
Since $|e^{i\pi |x|^{2}\cot\alpha}|=1$, we have $|\Ma f(x)|=|f(x)|$ almost everywhere. All assertions follow immediately.
\end{proof}

\begin{proposition}[\textbf{Definition of FrFT multipliers}]\label{prop:frft-mult}
Let $m_\alpha(\xi):=m((\sin\alpha)\xi)$,
then we have that for $1\le p\le \infty$,
$$T_m^\alpha=\Ma^{-1}T_{m_\alpha}\Ma,\quad\text{and}\quad\|T_m^\alpha\|_{L^p\to L^p}=\|T_{m_\alpha}\|_{L^p\to L^p}.$$
\end{proposition}

\begin{proof}
Under the calssical Fourier integral, the FrFT operator is
$$(\Fa f)(u)=c_{\alpha,n}\,e^{i\pi |u|^2\cot\alpha}\,\widehat{\Ma f}(u\csc\alpha),$$
where $c_{\alpha,n}\neq 0$ depends only on $(\alpha,n)$ and is irrelevant for boundedness arguments. Let
$$h=T_m^\alpha f.$$
By Definition $\ref{definition-2.3}$, we have
$$\Fa h(u)=m(u)\Fa f(u).$$
Substituting the factorization for both $\Fa h$ and $\Fa f$ yields
$$c_{\alpha,n} e^{i\pi |u|^2\cot\alpha}\widehat{\Ma h}(u\csc\alpha)=
m(u)\,c_{\alpha,n} e^{i\pi |u|^2\cot\alpha}\widehat{\Ma f}(u\csc\alpha).$$
After cancellation of the common nonzero factor,
$$
\widehat{\Ma h}(u\csc\alpha)=m(u)\widehat{\Ma f}(u\csc\alpha).
$$

Now put $\xi=u\csc\alpha$, i.e. $u=(\sin\alpha)\xi.$ Then we have
$$\widehat{\Ma h}(\xi)=m((\sin\alpha)\xi)\widehat{\Ma f}(\xi)=m_\alpha(\xi)\widehat{\Ma f}(\xi).$$
Therefore,
$$\Ma h=T_{m_\alpha}(\Ma f),$$
which is equivalent to
$$T_m^\alpha=\Ma^{-1}T_{m_\alpha}\Ma.$$
Taking operator norms and applying Lemma~\ref{lem:isometry}, we have
$$\|T_m^\alpha\|_{L^p\to L^p}=\|\Ma^{-1}T_{m_\alpha}\Ma\|_{L^p\to L^p}=\|T_{m_\alpha}\|_{L^p\to L^p}.$$
This finishes the proof of Lemma $\ref{prop:frft-mult}$.
\end{proof}

\begin{proposition}[\textbf{General principle for FrFT oscillation-type estimates}]\label{prop:general}
Let $A(g)$ and $B(g)$ be quantities built from pointwise absolute values, distribution functions, local averages, square-function magnitudes, sharp maximal operators, weak-$L^{p}$ norms, mixed $L^{p}(\ell^{r})$ norms, or dual pairings depending only on the modulus. Define
$$
A_{\alpha}(f):=A(\Ma f),\qquad B_{\alpha}(f):=B(\Ma f).
$$
If a classical estimate
$$
A(g)\le C B(g)
$$
holds for all admissible $g$, then
$$
A_{\alpha}(f)\le C B_{\alpha}(f)
$$
holds for all admissible $f$ with the same constant.
\end{proposition}

\begin{proof}
Set $g=\Ma f$. Then $A_{\alpha}(f)=A(g)$ and $B_{\alpha}(f)=B(g)$. The classical estimate gives the FrFT-side estimate immediately. Lemma~\ref{lem:isom} ensures that every building block listed is unchanged by multiplication with the unimodular chirp.
\end{proof}

\begin{proposition}[\textbf{Scalar, vector-valued, square-function, and maximal bounds}]\label{prop:family}
Let $\{S_j\}_{j\in J}$ be a family of operators and define $S_j^\alpha:=\Ma^{-1}S_j\Ma$.

\begin{enumerate}[label=(\roman*),leftmargin=2em]
\item If the following estimates holds uniformly in  $j$,
$$\norm{S_j g}_{L^p}\le A\norm{g}_{L^p},$$
then we have
$$\norm{S_j^\alpha f}_{L^p}\le A\norm{f}_{L^p}.$$

\item If $S$ is weak type $(1,1)$,
$$|\{x:|Sg(x)|>\lambda\}|\le \frac{A}{\lambda}\norm{g}_{L^1},$$
then $S^\alpha$ is weak type $(1,1)$ with the same constant.

\item If the following estimates holds,
$$\left\|\left(\sum_{j\in J}|S_j g_j|^2\right)^{1/2}\right\|_{L^p}\le A\left\|\left(\sum_{j\in J}|g_j|^2\right)^{1/2}\right\|_{L^p},$$
then we have
$$\left\|\left(\sum_{j\in J}|S_j^\alpha f_j|^2\right)^{1/2}\right\|_{L^p}\le A\left\|\left(\sum_{j\in J}|f_j|^2\right)^{1/2}\right\|_{L^p}.$$

\item If the following estimates holds,
$$\left\|\sup_{j\in J}|S_j g|\right\|_{L^p}\le A\norm{g}_{L^p},$$
then we have
$$\left\|\sup_{j\in J}|S_j^\alpha f|\right\|_{L^p}\le A\norm{f}_{L^p}.$$

\item If the following estimates holds,
$$\left\|\left(\int_0^\infty |S_t g|^2\,d\mu(t)\right)^{1/2}\right\|_{L^p}\le A\norm{g}_{L^p},$$
then we have
$$\left\|\left(\int_0^\infty |S_t^\alpha f|^2\,d\mu(t)\right)^{1/2}\right\|_{L^p}\le A\norm{f}_{L^p}.$$
\end{enumerate}
\end{proposition}

\begin{proof}
The proof is the same in all cases and relies only on the unimodularity of the chirp factor. 

For $(i)$,
$$|S_j^\alpha f(x)|=\left|\Ma^{-1}S_j(\Ma f)(x)\right|=|S_j(\Ma f)(x)|.$$
Thus, we have
$$\norm{S_j^\alpha f}_{L^p}=\norm{S_j(\Ma f)}_{L^p}\le A\norm{\Ma f}_{L^p}=A\norm{f}_{L^p}.$$
For $(ii)$, the level sets agree,
$$\{x:|S^\alpha f(x)|>\lambda\}=\{x:|S(\Ma f)(x)|>\lambda\},$$
so the weak-type estimate transfers by Lemma~\ref{lem:isometry}. 

For $(iii)$, we have
$$\left(\sum_j |S_j^\alpha f_j|^2\right)^{1/2}=\left(\sum_j |S_j(\Ma f_j)|^2\right)^{1/2},$$
and the estimate follows by applying the assumed vector-valued bound to $g_j=\Ma f_j$. 

For $(iv)$,
$$\sup_j |S_j^\alpha f|=\sup_j |S_j(\Ma f)|,
$$
and again Lemma~\ref{lem:isometry} finishes the proof. The continuous square-function statement $(v)$ is identical:
$$\left(\int_0^\infty |S_t^\alpha f|^2\,d\mu(t)\right)^{1/2}=\left(\int_0^\infty |S_t(\Ma f)|^2\,d\mu(t)\right)^{1/2}.$$
This completes the  proof of Lemma $\ref{prop:family}$.
\end{proof}

\begin{lemma}[\textbf{Uniqueness of the low-frequency remainder}]\label{lem:poly}
Let $1\le p<\infty$, suppose two polynomials $Q_1,Q_2$ satisfy
$$f-\Ma^{-1}Q_1\in L^p(\R^n),\qquad f-\Ma^{-1}Q_2\in L^p(\R^n).$$
Then, we have $Q_1=Q_2$.
\end{lemma}

\begin{proof}
Subtracting the two relations gives
$$
\Ma^{-1}(Q_1-Q_2)\in L^p(\R^n).
$$
Multiplying by $\Ma$ and using Lemma~\ref{lem:isometry},$$Q_1-Q_2\in L^p(\R^n).
$$
A nonzero polynomial cannot belong to $L^p(\R^n)$ for any finite $p$, because its growth at infinity is at least polynomial on a set of positive measure. Hence $Q_1-Q_2\equiv 0$.
\end{proof}

\begin{lemma}[\textbf{General conjugation principle}]\label{lem:conjprin}
Let
$$\Ma f(x):=e^{i\pi |x|^2\cot\alpha}f(x),\qquad\alpha\notin \pi\mathbb Z.$$
For any linear operator $T$, define its FrFT-conjugated operator by
$$T^\alpha:=\Ma^{-1}T\Ma.$$
Then the following hold:
\begin{enumerate}[leftmargin=2em,label=(\arabic*)]
\item $\Ma$ and $\Ma^{-1}$ preserve $\mathcal S(\R^n)$ and $\mathcal S'(\R^n)$;
\item for every $1\le p\le\infty$, $\Ma$ is an isometry on $L^p(\R^n)$, namely
$$\|\Ma f\|_{L^p(\R^n)}=\|f\|_{L^p(\R^n)};$$
\item for every family $\{T_j\}$, if $T_j^\alpha=\Ma^{-1}T_j\Ma$, then for every $f$,
$$T_j^\alpha f=\Ma^{-1}\bigl(T_j(\Ma f)\bigr)\quad\text{and}\quad|T_j^\alpha f|=|T_j(\Ma f)|.$$
Consequently, every estimate formulated in terms of pointwise absolute values, inner $\ell^2$ or more general $\ell^r$ norms, outer $L^p$ norms, weak-type distribution functions, or duality pairings transfers from the calssical Fourier setting to its FrFT-conjugated operator with the same constant.
\end{enumerate}
\end{lemma}

\begin{proof}
The operator $\Ma$ is simply multiplication by the unimodular smooth chirp
$e^{i\pi |x|^2\cot\alpha}.$
Every derivative of this factor is of the form ``polynomial $\times$ chirp'', hence multiplication by it preserves Schwartz functions; by duality, both $\Ma$ and $\Ma^{-1}$ preserve tempered distributions as well.

Since$$|e^{i\pi |x|^2\cot\alpha}|=1,$$
for every $1\le p\le\infty$, we have
$$\|\Ma f\|_{L^p(\R^n)}=\|f\|_{L^p(\R^n)},\qquad\|\Ma^{-1}f\|_{L^p(\R^n)}=\|f\|_{L^p(\R^n)}.$$
Finally, by definition, we have
$$T_j^\alpha f=\Ma^{-1}T_j(\Ma f),$$
and therefore
$$|T_j^\alpha f|=|T_j(\Ma f)|.$$

Hence every square function, vector-valued expression, weak-type level set
$$\{x:\mathfrak S_\alpha(f)(x)>\lambda\},$$
or dual form built from these pointwise magnitudes is identical to the corresponding classical Fourier expression applied to $\Ma f$. This proves the claim.
\end{proof}

\begin{lemma}[\textbf{Uniqueness of chirp-modulated polynomials}]\label{lem:chirp-poly}
Let $1\le p<\infty$. If $Q$ is a polynomial and
$$e^{-i\pi |x|^2\cot\alpha}Q(x)\in L^p(\R^n),$$
then $Q\equiv 0$.
\end{lemma}

\begin{proof}
Indeed,
$$\bigl|e^{-i\pi |x|^2\cot\alpha}Q(x)\bigr|=|Q(x)|.$$
Hence, it is follows that $Q\in L^p(\R^n)$. A nonzero polynomial cannot belong to $L^p(\R^n)$, so necessarily $Q=0$.
\end{proof}

\section{The FrFT Littlewood-Paley theorem}

\begin{theorem}[FrFT Littlewood--Paley theorem]\label{thm:FrFTL-P}
\begin{enumerate}[leftmargin=2em,label=\textnormal{(\alph*)}]

Assume that $\Psi$ is an integrable $C^1$ function with mean zero,
$$\int_{\R^n}\Psi(x)\,dx=0,$$
and satisfying
\begin{equation}\label{eq:MVZ}
|\Psi(x)|+|\nabla\Psi(x)|\le B(1+|x|)^{-n-1}.
\end{equation}
Then there exists a constants $C_n$ such that for all $1<p<\infty$ and all $f\in L^p(\R^n)$ we have
$$\|S_\alpha(f)\|_{L^p(\R^n)}\le C_n\,B\,\max\{p,(p-1)^{-1}\}\,\|f\|_{L^p(\R^n)}.$$
Moreover, the weak type $(1,1)$ estimate holds:
$$\|S_\alpha(f)\|_{L^{1,\infty}(\R^n)}\le C_n' B\,\|f\|_{L^1(\R^n)}.$$

Conversely, $\Psi$ be a Schwartz function such that either $\widehat{\Psi}(0)=0$ and
$$\sum_{j\in\mathbb Z}|\widehat{\Psi}(2^{-j}\xi)|^2=1\quad for~all~\xi\in\R^n\backslash\{0\},$$
or $\widehat{\Psi}$ is compactly supported away from the originand
$$\sum_{j\in\mathbb Z}\widehat{\Psi}(2^{-j}\xi)=1\quad for~all~\xi\in\R^n\backslash\{0\},$$
then there is a constant $C_{n, \Psi}$, such that for any $f \in \mathcal S'(\R^n)$ with $S_\alpha(f)\in L^p(\R^n)$ in $L^p\left(\R^n\right)$ for some $1<p<\infty$, there exists a unique polynomial $Q$ such that
$$f(x)-e^{-i\pi |x|^2\cot\alpha}Q(x)\in L^p(\R^n),$$
and we have
$$\bigl\|f-e^{-i\pi |x|^2\cot\alpha}Q\bigr\|_{L^p(\R^n)}\le C_{n,\Psi}\,B\,\max\{p,(p-1)^{-1}\}\,\|S_\alpha(f)\|_{L^p(\R^n)}.$$
In particular, if $f\in L^p(\R^n)$ for some $1<p<\infty$, then
$$\|f\|_{L^p(\R^n)}\approx \|S_\alpha(f)\|_{L^p(\R^n)}.$$
\end{enumerate}
\end{theorem}

\begin{proof}
The proof relies on Lemma~\ref{lem:conjprin}.

\textbf{Step 1: Square-function estimates and weak type $(1,1)$.} Let $ g:=\Ma f.$
Then, by definition we have
$$\Delta_j^\alpha f=\Ma^{-1}\Delta_j(\Ma f)=\Ma^{-1}\Delta_j g.$$
Hence, we obtain
$$|\Delta_j^\alpha f|=|\Delta_j g|,$$
and therefore
$$S_\alpha(f)=\left(\sum_{j\in\mathbb Z}|\Delta_j^\alpha f|^2\right)^{1/2}=\left(\sum_{j\in\mathbb Z}|\Delta_j g|^2\right)^{1/2}=:S(g).
$$
The classical Littlewood–Paley theorem immediately gives
$$\|S_\alpha(f)\|_{L^p(\R^n)}=\|S(g)\|_{L^p(\R^n)}
\le C_nB\max\{p,(p-1)^{-1}\}\|g\|_{L^p(\R^n)}.$$

Since $\|g\|_{L^p(\R^n)}=\|\Ma f\|_{L^p(\R^n)}=\|f\|_{L^p(\R^n)},$ the desired inequality follows that
$$\|S_\alpha(f)\|_{L^p(\R^n)}\le C_nB\max\{p,(p-1)^{-1}\}\|f\|_{L^p(\R^n)}.$$

The weak type $(1,1)$ bound holds because level sets coincide: for every $\lambda>0$,
$$\{x:S_\alpha(f)(x)>\lambda\}=\{x:S(g)(x)>\lambda\}.$$
Hence, we obtain
$$|\{x:S_\alpha(f)(x)>\lambda\}|=|\{x:S(g)(x)>\lambda\}|.$$
By the weak-type estimate, we have
$$|\{x:S(g)(x)>\lambda\}|\le \frac{C_n'B}{\lambda}\|g\|_{L^1(\R^n)}=\frac{C_n'B}{\lambda}\|f\|_{L^1(\R^n)},$$
which is equivalent to
$$\|S_\alpha(f)\|_{L^{1,\infty}(\R^n)}\le C_n'B\|f\|_{L^1(\R^n)}.$$

\textbf{Step 2: Converse estimate.} Assume that
$S_\alpha(f)\in L^p(\R^n).$ Since $S_\alpha(f)=S(g)$, the classical Littlewood--Paley theorem yields a unique polynomial $Q$ such that
$$g-Q\in L^p(\R^n),$$
with
$$\|g-Q\|_{L^p(\R^n)} \le C_{n,\Psi}B\max\{p,(p-1)^{-1}\}\|S(g)\|_{L^p(\R^n)}.$$
Applying $\Ma^{-1}$ to both sides gives
$$f-\Ma^{-1}Q\in L^p(\R^n),$$
and
$$\|f-\Ma^{-1}Q\|_{L^p(\R^n)}=\|g-Q\|_{L^p(\R^n)}\le C_{n,\Psi}B\max\{p,(p-1)^{-1}\}\|S_\alpha(f)\|_{L^p(\R^n)}.$$
Since $\Ma^{-1}Q(x)=e^{-i\pi |x|^2\cot\alpha}Q(x),$ we have
$$f-e^{-i\pi |x|^2\cot\alpha}Q\in L^p(\R^n).$$
Uniqueness follows from Lemma $\ref{lem:chirp-poly}$.
\end{proof}

\begin{proposition}[\textbf{Vactor-valued FrFT Littlewood--Paley theorem}]

Let $\Psi$ be an integrable $\mathcal C^1$ function on $\R^n$ with mean value zero that satisfies \eqref{eq:MVZ} and let $\Delta_j$ be the Littlewood-Paley operator associated with $\Psi$. Then there exists a constant $C_n<\infty$ such that for all $1<p, r<\infty$ and all sequences of $L^p$ functions $f_j$ we have
$$
\left\|\left(\sum_{j\in\mathbb Z}\left(\sum_{k\in\mathbb Z}|\Delta_k^\alpha(f_j)|^2\right)^{r/2}\right)^{1/r}\right\|_{L^p(\R^n)}
\le C_nB\,C_{p,r}^{\sharp}\left\|\left(\sum_{j\in\mathbb Z}|f_j|^r\right)^{1/r}\right\|_{L^p(\R^n)},$$
where
$$C_{p,r}^{\sharp}:=\max\{p,(p-1)^{-1}\}\max\{r,(r-1)^{-1}\}.$$

A corresponding weak-type version also holds. In particular,
$$\left\|\left(\sum_{j\in\mathbb Z}|\Delta_j^\alpha(f_j)|^r\right)^{1/r}\right\|_{L^p(\R^n)}\le C_nB\,C_{p,r}^{\sharp}\left\|
\left(\sum_{j\in\mathbb Z}|f_j|^r\right)^{1/r}\right\|_{L^p(\R^n)}.$$

\end{proposition}

\begin{proof}

Again let $g_j=\Ma f_j$. Then for every $j$, $k$,
$$\Delta_k^\alpha(f_j)=\Ma^{-1}\Delta_k(g_j),\qquad|\Delta_k^\alpha(f_j)|=|\Delta_k(g_j)|.$$
Hence, we have
$$\left(\sum_{j\in\mathbb Z}\left(\sum_{k\in\mathbb Z}|\Delta_k^\alpha(f_j)|^2\right)^{r/2}\right)^{1/r}=\left(\sum_{j\in\mathbb Z}
\left(\sum_{k\in\mathbb Z}|\Delta_k(g_j)|^2\right)^{r/2}\right)^{1/r}.$$
Applying Proposition 6.1.4 in \cite{grafakos-classical} to the right-hand side gives
$$\left\|\left(\sum_j\left(\sum_k|\Delta_k^\alpha(f_j)|^2\right)^{r/2}\right)^{1/r}\right\|_{L^p}\le
C_nB\,C_{p,r}^{\sharp}\left\|\left(\sum_j |g_j|^r\right)^{1/r}\right\|_{L^p}.$$
Since $|g_j|=|f_j|$, the desired estimate follows. The weak-type version is identical.
\end{proof}

\begin{definition}
For $j \in \mathbb Z$ we introduce the one-dimensional operator
$$\Delta_j^{\#}(f)(x)=\left(\widehat{f} \chi_{I_j}\right)^{\vee}(x),
$$
where
$$
I_j=\left[2^j, 2^{j+1}\right) \cup\left(-2^{j+1},-2^j\right],
$$
and $\Delta_j^{\#}$ is a version of the operator $\Delta_j$ in which the characteristic function of the set $2^j \leq|\xi|<2^{j+1}$ replaces the function $\widehat{\Psi}\left(2^{-j} \xi\right)$.
\end{definition}

\begin{theorem}[\textbf{Sharp dyadic interval and rectangle decompositions}]\label{theorem2}

In one dimension, we define
$$\Delta_j^{\#,\alpha}:=\Ma^{-1}\Delta_j^\#\Ma.$$
Then there exists a constant $C_1$, such that for every  $1<p<\infty$ and every $f\in L^p(\R)$, we have
$$\frac{1}{C_1\bigl(p+(p-1)^{-1}\bigr)^2}\|f\|_{L^p(\R)}\le\left\|
\left(\sum_{j\in\mathbb Z}|\Delta_j^{\#,\alpha}f|^2\right)^{1/2}\right\|_{L^p(\R)}\le C_1\bigl(p+(p-1)^{-1}\bigr)^2\|f\|_{L^p(\R)}.$$

\end{theorem}

\begin{proof}
Let $\Delta_j^\#$ denote a version of the operator $\Delta_j$ in which the characteristic function of the set $2^j \leq|\xi|<2^{j+1}$ replaces the function $\widehat{\Psi}\left(2^{-j} \xi\right)$, and define
$$\Delta_j^{\#,\alpha}:=\Ma^{-1}\Delta_j^\#\Ma.$$
Set $g=\Ma f$. Then we have
$$\Delta_j^{\#,\alpha}f=\Ma^{-1}\Delta_j^\# g,\qquad|\Delta_j^{\#,\alpha}f|=|\Delta_j^\# g|.$$
Thus, we have
$$\left(\sum_j |\Delta_j^{\#,\alpha}f|^2\right)^{1/2}=\left(\sum_j |\Delta_j^\# g|^2\right)^{1/2}.$$
Applying the Theorem~6.1.5 in \cite{grafakos-classical}, we obtain
$$\frac{1}{C_1(p+(p-1)^{-1})^2}\|g\|_{L^p(\R)}\le\left\|\left(\sum_j |\Delta_j^\# g|^2\right)^{1/2}\right\|_{L^p(\R)}\le C_1(p+(p-1)^{-1})^2\|g\|_{L^p(\R)}.
$$
Since $\|g\|_{L^p(\R)}=\|f\|_{L^p(\R)}$, this proves the FrFT version.
\end{proof}

\section{Fractional multiplier theory}
The classical multiplier theorems can extend to FrFT multiplier theorems once the rescaled symbol keeps the same assumptions.

\begin{lemma}[\textbf{Invariance of the one-dimensional Marcinkiewicz condition under a family of dilations}]\label{lem:marc1}
Let $m_\alpha(\xi)=m((\sin\alpha)\xi)$. If
$$\sup_{j\in\mathbb Z}\left(\int_{\sa2^j}^{\sa2^{j+1}} |m'(u)|\,du+\int_{-\sa2^{j+1}}^{-\sa2^j} |m'(u)|\,du\right)\le A,$$
then $m_\alpha$ satisfies the classical one-dimensional Marcinkiewicz condition with the same constant.
\end{lemma}

\begin{proof}
By the chain rule,
$$m_\alpha'(\xi)=(\sin\alpha)m'((\sin\alpha)\xi).$$
Hence, we have
$$\int_{2^j}^{2^{j+1}} |m_\alpha'(\xi)|\,d\xi=\int_{2^j}^{2^{j+1}} |(\sin\alpha)m'((\sin\alpha)\xi)|\,d\xi.$$
Set $u=(\sin\alpha)\xi$. If $\sin\alpha>0$, this becomes
$$\int_{\sa2^j}^{\sa2^{j+1}} |m'(u)|\,du.$$
If $\sin\alpha<0$, the substitution reverses orientation, but the absolute value turns the integral into the same quantity. The negative dyadic half-axis is identical. Therefore the total variation on every dyadic interval is preserved.
\end{proof}

\begin{lemma}[\textbf{Invariance of derivative and Mihlin–Hörmander conditions under a family of dilations}]\label{lem:mihlin}
Let $m_\alpha(\xi)=m((\sin\alpha)\xi)$ on $\R^n$.

\begin{enumerate}[leftmargin=2em,label=(\roman*)]
\item If 
$$|\partial_u^\beta m(u)|\le C_\beta |u|^{-|\beta|},$$
then we have
$$|\partial_\xi^\beta m_\alpha(\xi)|\le C_{\beta,\alpha}|\xi|^{-|\beta|}.$$

\item If
$$\left(\int_{R<|u|<2R} |\partial_u^\beta m(u)|^2\,du\right)^{1/2}\le A R^{\frac n2-|\beta|},$$
then we have
$$\left(\int_{R<|\xi|<2R} |\partial_\xi^\beta m_\alpha(\xi)|^2\,d\xi\right)^{1/2}
\le C_{\alpha,\beta} A R^{\frac n2-|\beta|}.$$
\end{enumerate}
\end{lemma}

\begin{proof}
For $(i)$, repeated chain differentiation gives
$$\partial_\xi^\beta m_\alpha(\xi)=(\sin\alpha)^{|\beta|}(\partial_u^\beta m)((\sin\alpha)\xi).$$
Hence, we have
$$|\partial_\xi^\beta m_\alpha(\xi)|\le|\sin\alpha|^{|\beta|} C_\beta |(\sin\alpha)\xi|^{-|\beta|}=C_\beta |\xi|^{-|\beta|}.$$

For $(ii)$, the same chain rule together with the change of variables $u=(\sin\alpha)\xi$ yields
\begin{align*}
\int_{R<|\xi|<2R} |\partial_\xi^\beta m_\alpha(\xi)|^2\,d\xi
&=
|\sin\alpha|^{2|\beta|}
\int_{R<|\xi|<2R} |(\partial_u^\beta m)((\sin\alpha)\xi)|^2\,d\xi \\
&=
|\sin\alpha|^{2|\beta|-n}
\int_{\sa R<|u|<2\sa R} |\partial_u^\beta m(u)|^2\,du.
\end{align*}
Taking square roots,
$$\left(\int_{R<|\xi|<2R} |\partial_\xi^\beta m_\alpha(\xi)|^2\,d\xi\right)^{1/2}
=|\sin\alpha|^{|\beta|-\frac n2}\left(\int_{\sa R<|u|<2\sa R} |\partial_u^\beta m(u)|^2\,du\right)^{1/2}.
$$
Applying the assumed bound at radius $\sa R$ gives
$$\left(\int_{R<|\xi|<2R} |\partial_\xi^\beta m_\alpha(\xi)|^2\,d\xi\right)^{1/2}\le|\sin\alpha|^{|\beta|-\frac n2}A(\sa R)^{\frac n2-|\beta|}=A R^{\frac n2-|\beta|}.$$
This finishes the proof of Lemma $\ref{lem:mihlin}$.
\end{proof}

Using the dyadic decomposition of $\R^n$, we can write any $L^{\infty}$ function $m$ as the sum
\begin{equation}\label{mj}
m_{\mathbf{j}}=\sum_{\mathbf{j} \in \mathbb{Z}^n} m \chi_{R_{\mathbf{j}}} \quad \text { a.e. },
\end{equation}
where $\mathbf{j}=\left(j_1, \ldots, j_n\right), R_{\mathbf{j}}=I_{j_1} \times \cdots \times I_{j_n}$, and $I_k=\left[2^k, 2^{k+1}\right) \bigcup\left(-2^{k+1},-2^k\right]$. For $\mathbf{j}$ in $\mathbb{Z}^n$ we set $m_{\mathbf{j}}=m \chi_{R_{\mathbf{j}}}$. A consequence of the ideas developed so far is the following characterization of $\mathcal M_p\left(\R^n\right)$ in terms of a vector-valued inequality.

\begin{theorem}[\textbf{Fractional vector-valued inequality}]\label{thm:frftvalue}
Let $m_{\mathbf j}$ be the operators defined in \eqref{mj}, then $T_m^\alpha$ is bounded on $L^p(\R^n)$ if and only if there exists a constant
$C_p>0$ such that for every family $\{f_{\mathbf j}\}_{\mathbf j\in\mathbb Z^n}$, we have
$$\left\|\left(\sum_{\mathbf j\in\mathbb Z^n}|T_{m_{\mathbf j}}^\alpha(f_{\mathbf j})|^2
\right)^{1/2}\right\|_{L^p(\R^n)}\le C_p\left\|\left(\sum_{\mathbf j\in\mathbb Z^n}|f_{\mathbf j}|^2\right)^{1/2}\right\|_{L^p(\R^n)}.$$
\end{theorem}

\begin{proof}
The proof has two parts. Firstly, we reduce the FrFT multiplier problem to the
Fourier multiplier problem. Secondly, we verify that the Marcinkiewicz,
derivative, and Mihlin--H\"ormander hypotheses are preserved, up to constants
depending only on $\alpha$ and the dimension, under the fixed rescaling
$\xi\mapsto (\sin\alpha)\xi.$

\textbf{Step 1: The conjugation identity.} By the transfer identity established earlier, we have
$$
T_m^\alpha=\Ma^{-1}T_{m_\alpha}\Ma,
\qquad
m_\alpha(\xi)=m((\sin\alpha)\xi).
$$
Hence, for every $1\le p\le\infty$, we have
$$\|T_m^\alpha f\|_{L^p(\R^n)}=\|T_{m_\alpha}(\Ma f)\|_{L^p(\R^n)}.$$
Since $\Ma$ is an isometry on every $L^p$, we have
$$\|T_m^\alpha\|_{L^p\to L^p}=\|T_{m_\alpha}\|_{L^p\to L^p}.$$
Likewise, for weak type $(1,1)$, we obtain
$$|\{x:|T_m^\alpha f(x)|>\lambda\}|=|\{x:|T_{m_\alpha}(\Ma f)(x)|>\lambda\}|.$$
Therefore, $L^p$ or weak-type estimate for $T_{m_\alpha}$ transfers
immediately to $T_m^\alpha$.

\textbf{Step 2: Invariance of the multiplier hypotheses under fixed rescaling.} For Proposition~6.2.1 in \cite{grafakos-classical}, each multiplier
$f\mapsto (\widehat f\,m_{\mathbf j})^\vee$
is replaced by its FrFT conjugate operator
$$T_{m_{\mathbf j}}^\alpha=\Ma^{-1}T_{(m_{\mathbf j})_\alpha}\Ma.
$$
If the classical vector-valued inequality holds, namely
$$\left\|\left(\sum_{\mathbf j}|T_{(m_{\mathbf j})_\alpha}g_{\mathbf j}|^2\right)^{1/2}\right\|_{L^p}
\le C_p\left\|\left(\sum_{\mathbf j}|g_{\mathbf j}|^2\right)^{1/2}\right\|_{L^p},$$
then by taking $g_{\mathbf j}=\Ma f_{\mathbf j}$ we obtain
$$\left\|\left(\sum_{\mathbf j}|T_{m_{\mathbf j}}^\alpha(f_{\mathbf j})|^2\right)^{1/2}\right\|_{L^p}
\le C_p\left\|\left(\sum_{\mathbf j}|f_{\mathbf j}|^2\right)^{1/2}\right\|_{L^p}.$$
The converse implication is identical. This completes the proof.
\end{proof}

\begin{theorem}[\textbf{Fractional Marcinkiewicz multiplier theorem on $\R$}]\label{thm:frftmar}

Let $m:\R\to\R$ be bounded that is $\mathcal C^1$ in every dyadic set
$$(2^j,2^{j+1})\cup(-2^{j+1},-2^j),\qquad j\in\mathbb Z.$$
Assume that the derivative $m^{\prime}$ of $m$ satisfies
$$\sup_j\left(\int_{-2^j}^{-2^{j+1}}|m'(\xi)|\,d\xi+\int_{2^j}^{2^{j+1}}|m'(\xi)|\,d\xi\right)\le A<\infty.$$
Then for every $1<p<\infty$, we have that $m \in \mathcal M_p(\R)$ and for some $C>0$ we have
$$\|T_m^\alpha f\|_{L^p(\R)}\le C_{\alpha,p}\bigl(\|m\|_{L^\infty(\R)}+A\bigr)\|f\|_{L^p(\R)},$$
where
$$
C_{\alpha,p}\lesssim C_{\alpha} \max\{p,(p-1)^{-1}\}^{6}.
$$
\end{theorem}

\begin{proof}
Let $m_\alpha(\xi)=m((\sin\alpha)\xi).$ Then, we have
$$m_\alpha'(\xi)=(\sin\alpha)\,m'((\sin\alpha)\xi).$$
Thus, for every dyadic interval $(2^j,2^{j+1})$, we have
$$\int_{2^j}^{2^{j+1}}|m_\alpha'(\xi)|\,d\xi=\int_{(\sin\alpha)2^j}^{(\sin\alpha)2^{j+1}}|m'(u)|\,du,$$
after the change of variables $u=(\sin\alpha)\xi$. 

Since the image interval differs
from an classical dyadic interval only by a fixed scale factor, it can be covered by
finitely many dyadic intervals, with covering number depending only on
$\alpha$. It is follows that
$$\sup_j \int_{2^j}^{2^{j+1}}|m_\alpha'(\xi)|\,d\xi\le C_\alpha A,$$
and similarly on the negative side. It is easy to obtain that  $\|m_\alpha\|_{L^\infty}=\|m\|_{L^\infty}.$

Therefore the Marcinkiewicz multiplier theorem on $\R$ applies to $m_\alpha$, yielding
$$\|T_{m_\alpha}g\|_{L^p(\R)}\le C_{\alpha,p}\bigl(\|m\|_{L^\infty}+A\bigr)\|g\|_{L^p(\R)}.$$
Taking $g=\Ma f$ gives the desired FrFT estimate.
\end{proof}

\begin{theorem}[\textbf{Fractional Mihlin--H\"ormander condition}]\label{FMHC}
Let $m$ be a complex-valued bounded function on $\R^n\setminus\{0\}$ such that for every
multi-index $\beta$ with $|\beta|\le [n/2]+1$ and every $R>0$,
$$\left(\int_{R<|\xi|<2R}|\partial^\beta m(\xi)|^2\,d\xi\right)^{1/2}\le A\,R^{n/2-|\beta|}<\infty.$$
Then, for every $1<p<\infty$, $m$ lies in $\mathcal M(\R^n)$ and the following estiomat is valid:
$$\|T_m^\alpha f\|_{L^p(\R^n)}\le C_{\alpha,n}\max\{p,(p-1)^{-1}\}\bigl(A+\|m\|_{L^\infty(\R^n)}\bigr)\|f\|_{L^p(\R^n)},$$
and $T_m^\alpha$ is also of weak type $(1,1)$ with norm at most a constant multiple of
$C_{\alpha,n}(A+\|m\|_\infty)$.
\end{theorem}

\begin{proof}
For any multi-index $\beta$,
$$\partial^\beta m_\alpha(\xi)=(\sin\alpha)^{|\beta|}(\partial^\beta m)\bigl((\sin\alpha)\xi\bigr).$$
Hence, we have
$$\begin{aligned}\left(\int_{R<|\xi|<2R}|\partial^\beta m_\alpha(\xi)|^2\,d\xi\right)^{1/2}
&=|\sin\alpha|^{|\beta|}
\left(\int_{R<|\xi|<2R}\bigl|(\partial^\beta m)((\sin\alpha)\xi)\bigr|^2\,d\xi\right)^{1/2} \\
&=|\sin\alpha|^{|\beta|-n/2}\left(\int_{|\sin\alpha|R<|u|<2|\sin\alpha|R}|\partial^\beta m(u)|^2\,du\right)^{1/2} \\&\le
C_{\alpha,n}A\,R^{n/2-|\beta|}.
\end{aligned}
$$
Therefore $m_\alpha$ satisfies the classical Mihlin--H\"ormander condition. Applying the classical theorem to $m_\alpha$ gives
$$\|T_{m_\alpha}g\|_{L^p}\le C_{\alpha,n}\max\{p,(p-1)^{-1}\}(A+\|m\|_\infty)\|g\|_{L^p},$$
and also the weak-type estimate
$$\|T_{m_\alpha}g\|_{L^{1,\infty}}\le C_{\alpha,n}(A+\|m\|_\infty)\|g\|_{L^1}.$$
Taking $g=\Ma f$ and using the fact that $\Ma$ does not change level sets proves the FrFT version. This completes the proof.
\end{proof}

\section{Maximal bounds, rough square functions, and almost orthogonality}
In this section, we will give some positive maximal and rough-kernel theorems.

\begin{theorem}[\textbf{Transport of maximal multiplier bounds}]\label{thm:maximal-transport}
Let $\alpha\notin\pi\mathbb Z$, let $\{T_k\}_{k\in\mathbb Z}$ be a family of classical multiplier operators on $\R^n$, and define
$$T_k^\alpha:=\Mai T_k\Ma,\qquad k\in\mathbb Z.$$
Assume that for some $1\le p\le \infty$ and some constant $C>0$,
$$\left\|\sup_{k\in\mathbb Z}|T_k g|\right\|_{L^p(\R^n)}
\le C\|g\|_{L^p(\R^n)}
\qquad\text{for all } g\in L^p(\R^n).$$
Then, we have
$$\left\|\sup_{k\in\mathbb Z}|T_k^\alpha f|\right\|_{L^p(\R^n)}
\le C\|f\|_{L^p(\R^n)}
\qquad\text{for all } f\in L^p(\R^n).$$
Consequently, every positive maximal estimate for dilated multipliers, rectangular partial sums, or continuous average-square families has an FrFT analogue with the same constant.
\end{theorem}

\begin{proof}
Fix $f\in L^p(\R^n)$ and set $g:=\Ma f.$ By definition of the FrFT-conjugated operator, for every $k\in\mathbb Z$, we have 
$$T_k^\alpha f=\Mai T_k(\Ma f)=\Mai T_k g.$$
Since $\Mai$ is multiplication by the unimodular chirp $e^{-i\pi|x|^2\cot\alpha},$
we have that for almost every $x\in\R^n$,
$$|T_k^\alpha f(x)|=|\Mai T_k g(x)|=|T_k g(x)|=|T_k(\Ma f)(x)|.$$
Taking the supremum over $k\in\mathbb Z$, we obtain the pointwise identity
$$\sup_{k\in\mathbb Z}|T_k^\alpha f(x)|=\sup_{k\in\mathbb Z}|T_k(\Ma f)(x)|.$$
Therefore, we have
$$\left\|\sup_{k\in\mathbb Z}|T_k^\alpha f|\right\|_{L^p(\R^n)}=\left\|\sup_{k\in\mathbb Z}|T_k(\Ma f)|\right\|_{L^p(\R^n)}.$$

Applying the assumed classical maximal estimate with $g=\Ma f$, we get
$$\left\|\sup_{k\in\mathbb Z}|T_k^\alpha f|\right\|_{L^p(\R^n)}\le C\|\Ma f\|_{L^p(\R^n)}.$$
Now use the fact that $\Ma$ is an isometry on every $L^p(\R^n)$:
$$\|\Ma f\|_{L^p(\R^n)}=\|f\|_{L^p(\R^n)}.$$
Hence, we have
$$\left\|\sup_{k\in\mathbb Z}|T_k^\alpha f|\right\|_{L^p(\R^n)}\le C\|f\|_{L^p(\R^n)}.$$
This proves the theorem.
\end{proof}

\begin{theorem}[\textbf{Fractional of rough square-function bounds}]\label{thm:rough-square-transport}
Let $\alpha\notin\pi\mathbb Z$, let $\mu$ be a compactly supported finite Borel measure satisfying the usual Fourier decay condition, and let $\{T_{\mu,j}\}_{j\in\mathbb Z}$ be the associated classical dyadically dilated rough family. Define
$$
T_{\mu,j}^\alpha:=\Mai T_{\mu,j}\Ma,\qquad j\in\mathbb Z,
$$
and
$$G_\mu(g):=\left(\sum_{j\in\mathbb Z}|T_{\mu,j}g|^2\right)^{1/2},\qquad G_\mu^\alpha(f):=
\left(\sum_{j\in\mathbb Z}|T_{\mu,j}^\alpha f|^2\right)^{1/2}.
$$
Assume that for some $1<p<\infty$ and some constant $C>0$,
$$\|G_\mu(g)\|_{L^p(\R^n)}\le C\|g\|_{L^p(\R^n)}\qquad\text{for all } g\in L^p(\R^n).$$
Then, we have
$$\|G_\mu^\alpha(f)\|_{L^p(\R^n)}\le C\|f\|_{L^p(\R^n)}\qquad\text{for all } f\in L^p(\R^n).$$
Hence the FrFT rough square function is bounded on the same $L^p$ range with the same constant as the classical one.
\end{theorem}

\begin{proof}
Fix $f\in L^p(\R^n)$ and set $g:=\Ma f.$ For each $j\in\mathbb Z$, by definition of the FrFT-conjugated rough family,
$$T_{\mu,j}^\alpha f=\Mai T_{\mu,j}(\Ma f)=\Mai T_{\mu,j}g.$$
Since $\Mai$ is multiplication by a unimodular factor, we have that for almost every $x\in\R^n$ and every $j\in\mathbb Z$,
$$|T_{\mu,j}^\alpha f(x)|=|T_{\mu,j}g(x)|=|T_{\mu,j}(\Ma f)(x)|.$$

To justify the square-function identity carefully, for each $N\in\mathbb N$ define the truncated square functions
$$G_{\mu,N}(g):=\left(\sum_{|j|\le N}|T_{\mu,j}g|^2\right)^{1/2},\qquad G_{\mu,N}^\alpha(f):=\left(\sum_{|j|\le N}|T_{\mu,j}^\alpha f|^2\right)^{1/2}.$$
Then for almost every $x\in\R^n$,
$$G_{\mu,N}^\alpha(f)(x)=\left(\sum_{|j|\le N}|T_{\mu,j}^\alpha f(x)|^2\right)^{1/2}=\left(\sum_{|j|\le N}|T_{\mu,j}(\Ma f)(x)|^2\right)^{1/2}
=G_{\mu,N}(\Ma f)(x).$$
Taking $L^p$ norms, we obtain
$$\|G_{\mu,N}^\alpha(f)\|_{L^p(\R^n)}=\|G_{\mu,N}(\Ma f)\|_{L^p(\R^n)}.$$
Since $G_{\mu,N}(h)\le G_\mu(h)$ pointwise for every $h$, the assumed classical rough square-function estimate gives
$$\|G_{\mu,N}^\alpha(f)\|_{L^p(\R^n)}=\|G_{\mu,N}(\Ma f)\|_{L^p(\R^n)}\le \|G_\mu(\Ma f)\|_{L^p(\R^n)}\le C\|\Ma f\|_{L^p(\R^n)}.$$
Using again the $L^p$-isometry of $\Ma$, we have
$$\|G_{\mu,N}^\alpha(f)\|_{L^p(\R^n)}\le C\|f\|_{L^p(\R^n)}.$$
The constant on the right-hand side is independent of $N$.

Now let $N\to\infty$. Since
$$G_{\mu,N}^\alpha(f)(x)\uparrow G_\mu^\alpha(f)(x)\qquad\text{for almost every }x\in\R^n,$$
the monotone convergence theorem yields
$$\|G_\mu^\alpha(f)\|_{L^p(\R^n)}=\lim_{N\to\infty}\|G_{\mu,N}^\alpha(f)\|_{L^p(\R^n)}\le C\|f\|_{L^p(\R^n)}.$$
Therefore, we have
$$\|G_\mu^\alpha(f)\|_{L^p(\R^n)}\le C\|f\|_{L^p(\R^n)}.$$
This proves the theorem.
\end{proof}

\begin{theorem}[\textbf{Fractional almost orthogonality}]\label{thm:634}
Let $\alpha\notin\pi\mathbb Z$, $1<p\le 2\le q<\infty$. Let $\{T_j^\alpha\}_{j\in\mathbb Z}$ be a family of FrFT
multiplier operators satisfying
$$T_j^\alpha=\Ma^{-1}T_j\Ma,$$
where each $T_j$ is Fourier multiplier operator whose symbol is
supported in the dyadic annulus
$$2^{j-1}\le |\xi|\le 2^{j+1}.$$
If
$$\sup_j \|T_j^\alpha\|_{L^p(\R^n)\to L^q(\R^n)}<\infty,$$
then for each $f\in L^p(\R^n)$, the series
$$T^\alpha(f):=\sum_{j\in\mathbb Z}T_j^\alpha(f)$$
converges in $L^q(\R^n)$, and there exists $C_{\alpha,p,q,n}$ such that
$$\|T^\alpha\|_{L^p(\R^n)\to L^q(\R^n)}\le C_{\alpha,p,q,n}\sup_j\|T_j^\alpha\|_{L^p\to L^q}.$$
\end{theorem}

\begin{proof}
Assume $T_j^\alpha=\Ma^{-1}T_j\Ma$, where each Fourier multiplier $T_j$
has symbol supported in
$$2^{j-1}\le |\xi|\le 2^{j+1}.$$
If
$$\sup_j \|T_j^\alpha\|_{L^p\to L^q}=A<\infty,$$
then, because $\Ma$ is an isometry on $L^p$ and $L^q$,
$$\sup_j \|T_j\|_{L^p\to L^q}=A.$$
Hence, using the Theorem~6.3.6 in \cite{grafakos-classical}, for every $g\in L^p$, we have
$$\left\|\sum_j T_j(g)\right\|_{L^q}\le C_{p,q,n}A\|g\|_{L^p}.$$
Now, set $g=\Ma f$ and define
$$T^\alpha(f):=\Ma^{-1}\left(\sum_j T_j(\Ma f)\right)=\sum_j T_j^\alpha(f).$$
Then, we have 
$$\|T^\alpha(f)\|_{L^q}=\left\|\sum_j T_j(\Ma f)\right\|_{L^q}\le C_{p,q,n}A\|f\|_{L^p}.$$
Thus, the FrFT series converges in $L^q$ and satisfies the same type of estimate.
\end{proof}

\section{Mixed FrFT--dyadic operator}

In tihis section, we notice that the classical dyadic martingale/Haar system does not as an ordinary dyadic system but it becomes a \emph{twisted} dyadic system.

\begin{definition}[\textbf{Twisted dyadic operators}]
Let $E_k$ and $D_k$ be the classical dyadic conditional expectation and martingale difference operators. Define
$$E_k^{\alpha}:=\Mai E_k\Ma\quad\text{and}\quad\qquad D_k^{\alpha}:=\Mai D_k\Ma.$$
In one dimension, if $h_I$ is the classical Haar function on a dyadic interval $I$, define
$$h_I^{\alpha}:=\Mai h_I=e^{-i\pi |x|^2\cot\alpha}h_I(x).$$
\end{definition}

\begin{proposition}[\textbf{The properties of the twisted dyadic operators}]\label{prop:dyadicbasic}
For every $1\le p\le\infty$, we have
$$\|E_k^\alpha\|_{L^p\to L^p}=\|E_k\|_{L^p\to L^p}\quad\text{and}\quad\|D_k^\alpha\|_{L^p\to L^p}=\|D_k\|_{L^p\to L^p}.$$
On $L^2(\R^n)$, we have
$$(E_k^\alpha)^\ast=E_k^\alpha\quad\text{and}\quad(D_k^\alpha)^\ast=D_k^\alpha.$$
\end{proposition}

\begin{proof}
The $L^p$ norm equalities follow from Lemma~\ref{lem:isometry}:
$$\|E_k^\alpha f\|_{L^p}=\|E_k(\Ma f)\|_{L^p}\le\|E_k\|_{L^p\to L^p}\|f\|_{L^p},$$
and the reverse inequality is obtained by applying the same estimate to $\Ma^{-1}g$. The same argument applies to $D_k^\alpha$.

On $L^2$, $E_k$ and $D_k$ are self-adjoint, while $\Ma$ is unitary. Hence
$$(E_k^\alpha)^\ast=(\Ma^{-1}E_k\Ma)^\ast=\Ma^{-1}E_k^\ast \Ma=\Ma^{-1}E_k\Ma=E_k^\alpha,$$
and similarly for $D_k^\alpha$.
\end{proof}

\begin{proposition}[\textbf{Twisted Haar expansion}]\label{prop:haar}
In one dimension, for every locally integrable $f$ and every $k\in\mathbb Z$, we have
$$D_k^\alpha(f)=\sum_{I\in D_{k-1}} \langle f,h_I^\alpha\rangle h_I^\alpha,$$
and
$$\|D_k^\alpha(f)\|_{L^2(\R)}^2=\sum_{I\in D_{k-1}} |\langle f,h_I^\alpha\rangle|^2.$$
\end{proposition}

\begin{proof}
Let $g=\Ma f$. The classical dyadic identity gives
$$D_k(g)=\sum_{I\in D_{k-1}}\langle g,h_I\rangle h_I.$$
Conjugating by $\Ma^{-1}$, we have
$$D_k^\alpha(f)=\Ma^{-1}D_k(\Ma f)=\Ma^{-1}\left(\sum_{I\in D_{k-1}}\langle \Ma f,h_I\rangle h_I\right).$$
Because $\Ma$ is unitary on $L^2$,
$$\langle \Ma f,h_I\rangle=\langle f,\Ma^{-1}h_I\rangle=\langle f,h_I^\alpha\rangle.$$
Therefore, we have
$$
D_k^\alpha(f)=\sum_{I\in D_{k-1}}\langle f,h_I^\alpha\rangle h_I^\alpha.
$$
For the norm identity,
$$
\|D_k^\alpha(f)\|_{L^2}
=
\|D_k(\Ma f)\|_{L^2},
$$
and Parseval's identity for the classical Haar expansion yields the result.
\end{proof}

\begin{theorem}[\textbf{Twisted martingale decomposition}]\label{thm:decomp}
For every $f\in L^2(\R^n)$,
$$f=\sum_{k\in\mathbb Z} D_k^\alpha(f)$$
both almost everywhere and in $L^2(\R^n)$, and
$$\|f\|_{L^2(\R^n)}^2=\sum_{k\in\mathbb Z}\|D_k^\alpha(f)\|_{L^2(\R^n)}^2.$$
In one dimension, we have
$$f=\sum_{I\in D}\langle f,h_I^\alpha\rangle h_I^\alpha\quad\text{and}\quad
\|f\|_{L^2(\R)}^2=\sum_{I\in D} |\langle f,h_I^\alpha\rangle|^2.$$
\end{theorem}

\begin{proof}
Let $g:=\Ma f.$ Since $\Ma$ is unitary on $L^2(\R^n)$, we have $g\in L^2(\R^n)$ and
$$\|g\|_{L^2(\R^n)}=\|f\|_{L^2(\R^n)}.$$
By the classical martingale decomposition theorem, we know that
$$g=\sum_{k\in\mathbb Z} D_k(g)$$
both almost everywhere and in $L^2(\R^n)$, and moreover
$$\|g\|_{L^2(\R^n)}^2=\sum_{k\in\mathbb Z}\|D_k(g)\|_{L^2(\R^n)}^2.$$
For each $N\in\mathbb N$, define the partial sums
$$S_N g:=\sum_{|k|\le N}D_k(g).$$
Then, we have
$$S_N g\to g\qquad\text{in }L^2(\R^n)\text{ and almost everywhere.}$$
Apply $\Mai$ to this identity. Since $\Mai$ is linear, we obtain
$$\Mai(S_N g)=\sum_{|k|\le N}\Mai D_k(g).$$
Because $g=\Ma f$ and $D_k^\alpha=\Mai D_k\Ma$, this becomes
$$\Mai(S_N g)=\sum_{|k|\le N}\Mai D_k(\Ma f)=\sum_{|k|\le N}D_k^\alpha(f).$$

On the other hand, since $\Mai$ is unitary on $L^2(\R^n)$,
$$\Mai(S_N g)\to \Mai g=f\qquad\text{in }L^2(\R^n).$$
Therefore, we have
$$f=\sum_{k\in\mathbb Z}D_k^\alpha(f)\qquad\text{in }L^2(\R^n).$$
To obtain almost-everywhere convergence, note that for almost every $x\in\R^n$,
$$
S_N g(x)\to g(x).
$$
Since $\Mai$ is multiplication by the unimodular chirp $e^{-i\pi|x|^2\cot\alpha},$
we may multiply the pointwise convergence by this factor and get
$$
\sum_{|k|\le N}D_k^\alpha(f)(x)=\Mai(S_N g)(x)\longrightarrow\Mai g(x)=f(x)$$
for almost every $x\in\R^n$. Hence
$$f=\sum_{k\in\mathbb Z}D_k^\alpha(f)\qquad\text{almost everywhere.}$$

Next we prove the norm identity. Since
$$D_k^\alpha(f)=\Mai D_k(\Ma f)=\Mai D_k(g),$$
and $\Mai$ is unitary on $L^2(\R^n)$, we have
$$\|D_k^\alpha(f)\|_{L^2(\R^n)}=\|D_k(g)\|_{L^2(\R^n)}\qquad\text{for every }k\in\mathbb Z.$$
Therefore, we have
$$
\sum_{k\in\mathbb Z}\|D_k^\alpha(f)\|_{L^2(\R^n)}^2=\sum_{k\in\mathbb Z}\|D_k(g)\|_{L^2(\R^n)}^2=\|g\|_{L^2(\R^n)}^2=\|f\|_{L^2(\R^n)}^2.$$
In one dimension, Proposition~\ref{prop:haar} gives the twisted Haar expansion
$$D_k^\alpha(f)=\sum_{I\in D_{k-1}}\langle f,h_I^\alpha\rangle h_I^\alpha.$$
Summing over $k\in\mathbb Z$ and using the decomposition already proved, we obtain
$$f=\sum_{I\in D}\langle f,h_I^\alpha\rangle h_I^\alpha$$
in $L^2(\R)$ and almost everywhere. Since $\{h_I^\alpha\}_{I\in D}$ is the image of the
classical Haar system under the unitary operator $\Mai$, it is again an orthonormal system in
$L^2(\R)$. Hence Parseval's identity yields
$$\|f\|_{L^2(\R)}^2=\sum_{I\in D} |\langle f,h_I^\alpha\rangle|^2.$$
This completes the proof.
\end{proof}

\begin{theorem}[\textbf{Twisted dyadic square function}]\label{thm:dyadicsquare}
Define
$$S_{\mathrm{dy}}^\alpha(f):=\left(\sum_{k\in\mathbb Z}|D_k^\alpha(f)|^2\right)^{1/2}.$$
Then for every $1<p<\infty$ there exists a constant $c_{p,n}>0$ such that
$$\frac1{c_{p,n}}\|f\|_{L^p(\R^n)}\le \|S_{\mathrm{dy}}^\alpha(f)\|_{L^p(\R^n)}\le c_{p,n}\|f\|_{L^p(\R^n)}.$$
The constant is the same as in the classical dyadic square-function theorem.
\end{theorem}

\begin{proof}
Fix $f\in L^p(\R^n)$ and set $g:=\Ma f.$ Since $\Ma$ is an isometry on $L^p(\R^n)$, we have
$$\|g\|_{L^p(\R^n)}=\|f\|_{L^p(\R^n)}.$$
By definition of the twisted martingale difference operators,
$$
D_k^\alpha(f)=\Mai D_k(\Ma f)=\Mai D_k(g).
$$
Because $\Mai$ is multiplication by a unimodular factor, we obtain the pointwise identity
$$|D_k^\alpha(f)(x)|=|D_k(g)(x)|\qquad\text{for almost every }x\in\R^n.$$
Therefore,
$$S_{\mathrm{dy}}^\alpha(f)(x)
=\left(\sum_{k\in\mathbb Z}|D_k^\alpha(f)(x)|^2\right)^{1/2}=\left(\sum_{k\in\mathbb Z}|D_k(g)(x)|^2\right)^{1/2}=S_{\mathrm{dy}}(g)(x)$$
for almost every $x\in\R^n$, where
$$S_{\mathrm{dy}}(g):=\left(\sum_{k\in\mathbb Z}|D_k(g)|^2\right)^{1/2}$$
denotes the classical dyadic square function.

Taking $L^p$ norms, we get
$$\|S_{\mathrm{dy}}^\alpha(f)\|_{L^p(\R^n)}=\|S_{\mathrm{dy}}(g)\|_{L^p(\R^n)}.$$
Now apply the classical dyadic square-function theorem to $g$:
$$\frac1{c_{p,n}}\|g\|_{L^p(\R^n)}\le\|S_{\mathrm{dy}}(g)\|_{L^p(\R^n)}\le c_{p,n}\|g\|_{L^p(\R^n)}.$$
Substituting $g=\Ma f$ and using the $L^p$-isometry of $\Ma$, we obtain
$$
\frac1{c_{p,n}}\|f\|_{L^p(\R^n)}=
\frac1{c_{p,n}}\|g\|_{L^p(\R^n)}\le\|S_{\mathrm{dy}}(g)\|_{L^p(\R^n)}=\|S_{\mathrm{dy}}^\alpha(f)\|_{L^p(\R^n)}$$
and
$$\|S_{\mathrm{dy}}^\alpha(f)\|_{L^p(\R^n)}
=\|S_{\mathrm{dy}}(g)\|_{L^p(\R^n)}\le c_{p,n}\|g\|_{L^p(\R^n)}=c_{p,n}\|f\|_{L^p(\R^n)}.
$$
Combining the two inequalities proves
$$\frac1{c_{p,n}}\|f\|_{L^p(\R^n)}\le\|S_{\mathrm{dy}}^\alpha(f)\|_{L^p(\R^n)}\le c_{p,n}\|f\|_{L^p(\R^n)}.$$
Since the proof is obtained by the exact pointwise identity
$$S_{\mathrm{dy}}^\alpha(f)=S_{\mathrm{dy}}(\Ma f),$$
the constant is the same as in the classical dyadic square-function theorem. This completes the proof.
\end{proof}

\begin{theorem}[\textbf{Mixed FrFT--dyadic almost orthogonality}]\label{thm:mixed}
Let $\Delta_j$ be the classical Littlewood--Paley operators, define
$$\Delta_j^\alpha=\Mai\Delta_j\Ma.$$
Then for all $j,k\in\mathbb Z$, we have
$$\|D_k^\alpha\Delta_j^\alpha\|_{L^2\to L^2}=\|\Delta_j^\alpha D_k^\alpha\|_{L^2\to L^2}\le C2^{-\frac12|j-k|}.
$$
\end{theorem}

\begin{proof}
Recall that
$$D_k^\alpha=\Mai D_k\Ma,\qquad\Delta_j^\alpha=\Mai\Delta_j\Ma.$$
We first compute the composition $D_k^\alpha\Delta_j^\alpha$:
$$D_k^\alpha\Delta_j^\alpha=(\Mai D_k\Ma)(\Mai\Delta_j\Ma).$$
Since $\Ma\Mai=\mathrm{Id}$, this simplifies to
$$D_k^\alpha\Delta_j^\alpha=\Mai D_k\Delta_j\Ma.$$
Similarly,
$$\Delta_j^\alpha D_k^\alpha=(\Mai\Delta_j\Ma)(\Mai D_k\Ma)=\Mai\Delta_j D_k\Ma.$$

Now let $A$ be any bounded operator on $L^2(\R^n)$. Since $\Ma$ is unitary on $L^2(\R^n)$, we have
$$\|\Mai A\Ma\|_{L^2\to L^2}=\|A\|_{L^2\to L^2}.$$
Applying this with $A=D_k\Delta_j$, we obtain
$$\|D_k^\alpha\Delta_j^\alpha\|_{L^2\to L^2}=\|\Mai D_k\Delta_j\Ma\|_{L^2\to L^2}=\|D_k\Delta_j\|_{L^2\to L^2}.$$
Likewise, with $A=\Delta_j D_k$, we get
$$\|\Delta_j^\alpha D_k^\alpha\|_{L^2\to L^2}=\|\Mai \Delta_j D_k\Ma\|_{L^2\to L^2}=\|\Delta_j D_k\|_{L^2\to L^2}.$$
By the classical mixed dyadic--Littlewood--Paley almost orthogonality estimate, we have
$$
\|D_k\Delta_j\|_{L^2\to L^2}\le C2^{-\frac12|j-k|}\qquad\text{and}\qquad\|\Delta_j D_k\|_{L^2\to L^2}\le C2^{-\frac12|j-k|}.$$
Therefore, we have
$$\|D_k^\alpha\Delta_j^\alpha\|_{L^2\to L^2}\le C2^{-\frac12|j-k|}$$
and
$$\|\Delta_j^\alpha D_k^\alpha\|_{L^2\to L^2}\le C2^{-\frac12|j-k|}.$$
Hence, we have
$$\|D_k^\alpha\Delta_j^\alpha\|_{L^2\to L^2}=\|\Delta_j^\alpha D_k^\alpha\|_{L^2\to L^2}\le C2^{-\frac12|j-k|}.$$
This proves the theorem.
\end{proof}

\section{Backbone and reconstructibility}

\begin{proposition}[\textbf{Results of multiplier and dyadic modules}]\label{prop:multiplier and dyadic modules}
Let $1\le p\le \infty$ and $1\le s\le \infty$. Let $T$ be a classical linear operator bounded on
$L^{p,s}(\R^n)$; in particular, $T$ may be a Fourier multiplier, a low-frequency block, or a
dyadic block. Define
$$T_\alpha:=\Mai T\Ma.$$
Then $T_\alpha$ is bounded on $L^{p,s}(\R^n)$ and
$$\|T_\alpha\|_{L^{p,s}\to L^{p,s}}=\|T\|_{L^{p,s}\to L^{p,s}}.$$
In particular, $T_\alpha$ has the same operator norm as $T$ on the corresponding Lebesgue and Lorentz scales.
\end{proposition}

\begin{proof}
Since
$$
\Ma f(x)=e^{i\pi |x|^2\cot\alpha}f(x),
\qquad
\Mai f(x)=e^{-i\pi |x|^2\cot\alpha}f(x),
$$
we have that for almost every $x\in\R^n$,
$$
|\Ma f(x)|=|f(x)|
\quad \text{and}\quad
|\Mai f(x)|=|f(x)|.
$$
Hence $\Ma$ and $\Mai$ preserve distribution functions, and therefore
they are isometries on every $L^{p,s}(\R^n)$:
$$
\|\Ma f\|_{L^{p,s}}=\|f\|_{L^{p,s}},
\qquad
\|\Mai f\|_{L^{p,s}}=\|f\|_{L^{p,s}}.
$$

Now let $f\in L^{p,s}(\R^n)$. By definition,
$$
T_\alpha f=\Mai T(\Ma f).
$$
Therefore, we have
$$
\|T_\alpha f\|_{L^{p,s}}=\|\Mai T(\Ma f)\|_{L^{p,s}}=\|T(\Ma f)\|_{L^{p,s}}
\le
\|T\|_{L^{p,s}\to L^{p,s}}\|\Ma f\|_{L^{p,s}}=\|T\|_{L^{p,s}\to L^{p,s}}\|f\|_{L^{p,s}}.
$$
Thus
$$
\|T_\alpha\|_{L^{p,s}\to L^{p,s}}
\le
\|T\|_{L^{p,s}\to L^{p,s}}.
$$

Conversely, since
$$
T=\Ma T_\alpha \Mai,
$$
the same argument gives
$$
\|T\|_{L^{p,s}\to L^{p,s}}
\le
\|T_\alpha\|_{L^{p,s}\to L^{p,s}}.
$$
Hence, we have
$$
\|T_\alpha\|_{L^{p,s}\to L^{p,s}}
=
\|T\|_{L^{p,s}\to L^{p,s}}.
$$
This proves the proposition.
\end{proof}

\begin{theorem}[\textbf{Reconstructible FrFT representation}]\label{thm:reconstructible-frft}\label{thm:reconstruct}
Assume that the classical dyadic pair $(\Phi,\Psi)$ satisfies the inhomogeneous Calder\'on
reproducing formula. Then for every tempered distribution $u$,
$$
S_{0,\alpha}u+\sum_{j\ge 1}\Delta_{j,\alpha}u=u
\qquad \text{in }\mathcal S'(\R^n).
$$
If one works in the homogeneous quotient $\mathcal S'(\R^n)/\mathcal P_\alpha$, then
$$
\sum_{j\in\mathbb Z}\Delta_{j,\alpha}u=u.
$$
\end{theorem}

\begin{proof}
Recall that
$$
S_{0,\alpha}:=\Mai S_0\Ma,
\qquad
\Delta_{j,\alpha}:=\Mai\Delta_j\Ma.
$$
Let $v:=\Ma u.$ Since $\Ma$ is multiplication by a smooth unimodular chirp, it preserves $\mathcal S(\R^n)$ and
hence, by duality, also preserves $\mathcal S'(\R^n)$. Thus $v\in\mathcal S'(\R^n)$.

By the classical inhomogeneous Calder\'on reproducing formula,
$$
S_0v+\sum_{j\ge 1}\Delta_j v=v
\qquad \text{in }\mathcal S'(\R^n).
$$
Applying $\Mai$ to both sides and using linearity, we obtain
$$
\Mai S_0\Ma\,u+\sum_{j\ge 1}\Mai\Delta_j\Ma\,u=u,
$$
that is,
$$
S_{0,\alpha}u+\sum_{j\ge 1}\Delta_{j,\alpha}u=u
\qquad \text{in }\mathcal S'(\R^n).
$$
For the homogeneous statement, the classical Calder\'on formula holds modulo polynomials:
$$
\sum_{j\in\mathbb Z}\Delta_j v=v
\qquad \text{in }\mathcal S'(\R^n)/\mathcal P.
$$
Applying this to $v=\Ma u$ and then conjugating by $\Mai$, we obtain
$$
\sum_{j\in\mathbb Z}\Delta_{j,\alpha}u=u
\qquad \text{in }\mathcal S'(\R^n)/\Mai\mathcal P.
$$

Define
$$
\mathcal P_\alpha:=\Mai\mathcal P.
$$
Then, we have
$$
\sum_{j\in\mathbb Z}\Delta_{j,\alpha}u=u
\qquad \text{in }\mathcal S'(\R^n)/\mathcal P_\alpha.
$$
Thus, the FrFT component family reconstructs the input in the inhomogeneous setting, and
up to the natural polynomial ambiguity in the homogeneous setting. This proves the
theorem.
\end{proof}

\section{FrFT pullback spaces, Riesz--Bessel operators, and supporting estimates}

\begin{definition}[\textbf{Pullback spaces and core operators}]
For $1\le r\le\infty$, define
$$
L^r_{\alpha}(\R^n)
:=
\bigl\{f\in\mathcal S'(\R^n):F_{\alpha}f\in L^r(\R^n)\bigr\},
\qquad
\|f\|_{L^r_\alpha}:=\|F_\alpha f\|_{L^r}.
$$
For $1\le r<\infty$, define
$$
L^{r,\infty}_{\alpha}(\R^n)
:=
\bigl\{f\in\mathcal S'(\R^n):F_{\alpha}f\in L^{r,\infty}(\R^n)\bigr\},
\qquad
\|f\|_{L^{r,\infty}_\alpha}:=\|F_\alpha f\|_{L^{r,\infty}}.
$$
Define the FrFT-side fractional Laplacian, Riesz potential, and Bessel potential by
$$
(-\Delta_{\alpha})^{z/2}:=F_\alpha^{-1}(-\Delta)^{z/2}F_\alpha,
\qquad
I_{s,\alpha}:=F_\alpha^{-1}I_sF_\alpha,
\qquad
J_{\sigma,\alpha}:=F_\alpha^{-1}J_\sigma F_\alpha.
$$
\end{definition}

\begin{theorem}[\textbf{FrFT principle on pullback scales}]\label{thm:pullback-scales}
Let $1\le p,q\le\infty$. If a classical operator $S$ is bounded from $L^p(\R^n)$ to $L^q(\R^n)$, or from $L^p(\R^n)$ to $L^{q,\infty}(\R^n)$, then
$$
S_{\alpha}:=F_{\alpha}^{-1}SF_{\alpha}
$$
is bounded from $L^p_{\alpha}(\R^n)$ to $L^q_{\alpha}(\R^n)$, or from $L^p_{\alpha}(\R^n)$ to $L^{q,\infty}_{\alpha}(\R^n)$, with the same operator norm.
\end{theorem}

\begin{proof}
By definition of the pullback norms,
$$
\|f\|_{L^r_\alpha}=\|F_\alpha f\|_{L^r},
\qquad
\|f\|_{L^{r,\infty}_\alpha}=\|F_\alpha f\|_{L^{r,\infty}}.
$$
Thus, $F_\alpha$ is an isometric isomorphism from $L^r_\alpha$ onto $L^r$, and from
$L^{r,\infty}_\alpha$ onto $L^{r,\infty}$.

Let $f\in L^p_\alpha(\R^n)$ and set
$$
g:=F_\alpha f.
$$
Then, we have
$$
F_\alpha(S_\alpha f)=SF_\alpha f=Sg.
$$
Hence, in the strong-type case,
$$
\|S_\alpha f\|_{L^q_\alpha}
=
\|F_\alpha(S_\alpha f)\|_{L^q}
=
\|Sg\|_{L^q}
\le
\|S\|_{L^p\to L^q}\|g\|_{L^p}
=
\|S\|_{L^p\to L^q}\|f\|_{L^p_\alpha}.
$$
Likewise, in the weak-type case,
$$
\|S_\alpha f\|_{L^{q,\infty}_\alpha}
=
\|F_\alpha(S_\alpha f)\|_{L^{q,\infty}}
=
\|Sg\|_{L^{q,\infty}}
\le
\|S\|_{L^p\to L^{q,\infty}}\|g\|_{L^p}
=
\|S\|_{L^p\to L^{q,\infty}}\|f\|_{L^p_\alpha}.
$$
The reverse norm inequality follows from
$$
S=F_\alpha S_\alpha F_\alpha^{-1}
$$
by the same argument. Therefore the FrFT operator has the same norm.
\end{proof}

\begin{proposition}[Semigroup and commutation identities]\label{prop:semigroup-commutation}
Whenever all quantities are meaningful,
$$
I_{s,\alpha}I_{t,\alpha}=I_{s+t,\alpha},\qquad
J_{\sigma,\alpha}J_{\tau,\alpha}=J_{\sigma+\tau,\alpha},\qquad
(-\Delta_{\alpha})^{z/2}(-\Delta_{\alpha})^{w/2}=(-\Delta_{\alpha})^{(z+w)/2},
$$
and if $\Re s>2\Re z$, then
$$
I_{s,\alpha}(-\Delta_{\alpha})^z
=
(-\Delta_{\alpha})^zI_{s,\alpha}
=
I_{s-2z,\alpha}.
$$
\end{proposition}

\begin{proof}
By definition,
$$
I_{s,\alpha}=F_\alpha^{-1}I_sF_\alpha,\qquad
J_{\sigma,\alpha}=F_\alpha^{-1}J_\sigma F_\alpha,\qquad
(-\Delta_\alpha)^{z/2}=F_\alpha^{-1}(-\Delta)^{z/2}F_\alpha.
$$
Therefore
$$
I_{s,\alpha}I_{t,\alpha}
=
F_\alpha^{-1}I_sF_\alpha F_\alpha^{-1}I_tF_\alpha
=
F_\alpha^{-1}I_sI_tF_\alpha,
$$
and similarly
$$
J_{\sigma,\alpha}J_{\tau,\alpha}
=
F_\alpha^{-1}J_\sigma J_\tau F_\alpha,
\qquad
(-\Delta_\alpha)^{z/2}(-\Delta_\alpha)^{w/2}
=
F_\alpha^{-1}(-\Delta)^{z/2}(-\Delta)^{w/2}F_\alpha.
$$
Thus the first three identities follow immediately from the corresponding classical identities
$$
I_sI_t=I_{s+t},\qquad
J_\sigma J_\tau=J_{\sigma+\tau},\qquad
(-\Delta)^{z/2}(-\Delta)^{w/2}=(-\Delta)^{(z+w)/2}.
$$

Likewise,
$$
I_{s,\alpha}(-\Delta_\alpha)^z
=
F_\alpha^{-1}I_s(-\Delta)^zF_\alpha,
\qquad
(-\Delta_\alpha)^zI_{s,\alpha}
=
F_\alpha^{-1}(-\Delta)^zI_sF_\alpha.
$$
Since, on the classical side,
$$
I_s(-\Delta)^z=(-\Delta)^zI_s=I_{s-2z}
\qquad
(\Re s>2\Re z),
$$
we obtain
$$
I_{s,\alpha}(-\Delta_\alpha)^z
=
(-\Delta_\alpha)^zI_{s,\alpha}
=
I_{s-2z,\alpha}.
$$
This proves the proposition.
\end{proof}

\begin{theorem}[\textbf{FrFT Hardy-Littlewood-Sobolev theorem and Bessel regularization}]\label{thm:frft-hls-bessel}
Let $0<s<n$.
If $1<p<q<\infty$ and
$$
\frac1p-\frac1q=\frac{s}{n},
$$
then
$$
\|I_{s,\alpha}f\|_{L^q_{\alpha}}\le C\|f\|_{L^p_{\alpha}}.
$$
The endpoint weak estimate also holds:
$$
\|I_{s,\alpha}f\|_{L^{\frac{n}{n-s},\infty}_{\alpha}}
\le
C\|f\|_{L^1_{\alpha}}.
$$
Moreover,
$$
\|J_{s,\alpha}f\|_{L^r_{\alpha}}\le\|f\|_{L^r_{\alpha}}
\qquad
\text{for every }1\le r\le\infty.
$$
\end{theorem}

\begin{proof}
By definition, we have
$$
I_{s,\alpha}=F_\alpha^{-1}I_sF_\alpha,
\qquad
J_{s,\alpha}=F_\alpha^{-1}J_sF_\alpha.
$$
The classical Hardy--Littlewood--Sobolev theorem gives
$$
I_s:L^p(\R^n)\to L^q(\R^n)
\qquad
\text{if }
\frac1p-\frac1q=\frac{s}{n},
$$
and also the endpoint weak estimate
$$
I_s:L^1(\R^n)\to L^{\frac{n}{n-s},\infty}(\R^n).
$$
On the other hand, the classical Bessel potential operator satisfies
$$
\|J_s g\|_{L^r(\R^n)}\le \|g\|_{L^r(\R^n)}
\qquad
\text{for every }1\le r\le\infty.
$$
Applying Theorem~\ref{thm:pullback-scales} to these classical estimates yields
$$
\|I_{s,\alpha}f\|_{L^q_\alpha}\le C\|f\|_{L^p_\alpha},
\qquad
\|I_{s,\alpha}f\|_{L^{\frac{n}{n-s},\infty}_\alpha}\le C\|f\|_{L^1_\alpha},
$$
and
$$
\|J_{s,\alpha}f\|_{L^r_\alpha}\le \|f\|_{L^r_\alpha}
\qquad
(1\le r\le\infty).
$$
This proves the theorem.
\end{proof}

\begin{proposition}[\textbf{Operator-chain estimate for unified representation}]\label{prop:operator-chain-unified}
Let $0<s<n$ and let $1<p<q<\infty$ satisfy
$$
\frac1p-\frac1q=\frac{s}{n}.
$$
Assume that
$$
T_m:X_m\to L^p_{\alpha}(\R^n)
$$
is bounded with
$$
\|T_m f\|_{L^p_\alpha(\R^n)}
\le
C_m\|f\|_{X_m},
$$
and that
$$
A_{\alpha}:L^q_{\alpha}(\R^n)\to L^q_{\alpha}(\R^n)
$$
is bounded with
$$
\|A_\alpha h\|_{L^q_\alpha(\R^n)}
\le
C_A\|h\|_{L^q_\alpha(\R^n)}.
$$
For $\sigma>0$, define
$$
U_mf:=J_{\sigma,\alpha}A_{\alpha}I_{s,\alpha}T_mf.
$$
Then, we have
$$
\|U_mf\|_{L^q_{\alpha}(\R^n)}
\le
C\,C_m\|f\|_{X_m},
$$
where one may take $ C=C_A\,C_{n,s,p}.$
\end{proposition}

\begin{proof}
Let $g:=T_mf.$ Then $g\in L^p_\alpha(\R^n)$ and, by the boundedness of $T_m$, we have
$$
\|g\|_{L^p_\alpha(\R^n)}
=
\|T_mf\|_{L^p_\alpha(\R^n)}
\le
C_m\|f\|_{X_m}.
$$
By Theorem~\ref{thm:frft-hls-bessel},
since
$$
\frac1p-\frac1q=\frac{s}{n},
$$
we have
$$
\|I_{s,\alpha}g\|_{L^q_\alpha(\R^n)}
\le
C_{n,s,p}\|g\|_{L^p_\alpha(\R^n)}.
$$
Applying the bounded operator $A_\alpha$ on $L^q_\alpha(\R^n)$ yields
$$
\|A_\alpha I_{s,\alpha}g\|_{L^q_\alpha(\R^n)}
\le
C_A\|I_{s,\alpha}g\|_{L^q_\alpha(\R^n)}
\le
C_A C_{n,s,p}\|g\|_{L^p_\alpha(\R^n)}.
$$
Now apply again Theorem~\ref{thm:frft-hls-bessel}, this time to the Bessel regularization
operator at exponent $q$. Since $J_{\sigma,\alpha}$ is contractive on every $L^r_\alpha$,
in particular on $L^q_\alpha(\R^n)$, we obtain
$$
\|J_{\sigma,\alpha}A_\alpha I_{s,\alpha}g\|_{L^q_\alpha(\R^n)}
\le
\|A_\alpha I_{s,\alpha}g\|_{L^q_\alpha(\R^n)}.
$$
Therefore
$$
\|U_mf\|_{L^q_\alpha(\R^n)}
=
\|J_{\sigma,\alpha}A_\alpha I_{s,\alpha}T_mf\|_{L^q_\alpha(\R^n)}
\le
C_A C_{n,s,p}\|T_mf\|_{L^p_\alpha(\R^n)}.
$$
Using once more the boundedness of $T_m$, we conclude that
$$
\|U_mf\|_{L^q_\alpha(\R^n)}
\le
C_A C_{n,s,p} C_m \|f\|_{X_m}.
$$
Hence, the desired estimate holds with $C=C_A C_{n,s,p}.$
This proves the proposition.
\end{proof}

\section{Riesz-Bessel regularization as a supporting layer}
For robust representation one may append the FrFT Riesz and Bessel operators
$$
I_{s,\alpha}:=F_{\alpha}^{-1}\bigl((2\pi|\xi|)^{-s}F_\alpha\cdot\bigr),\qquad
J_{\sigma,\alpha}:=F_{\alpha}^{-1}\bigl((1+4\pi^2|\xi|^2)^{-\sigma/2}F_\alpha\cdot\bigr).
$$

\begin{theorem}[\textbf{Supporting regularization estimate}]\label{thm:supporting-regularization}\label{thm:hls}
Let $0<s<n$ and let $1<p<q<\infty$ satisfy
$$
\frac1p-\frac1q=\frac{s}{n}.
$$
Then, we have
$$
\|I_{s,\alpha}f\|_{L^q_\alpha}\le C\|f\|_{L^p_\alpha},
\qquad
\|J_{\sigma,\alpha}f\|_{L^p_\alpha}\le \|f\|_{L^p_\alpha}.
$$
\end{theorem}

\begin{proof}
The classical Hardy--Littlewood--Sobolev theorem yields
$$
\|I_s g\|_{L^q(\R^n)}\le C\|g\|_{L^p(\R^n)}
\qquad
\text{whenever }
\frac1p-\frac1q=\frac{s}{n}.
$$
Also, the Bessel potential operator $J_\sigma$ is bounded on $L^p(\R^n)$ with
$$
\|J_\sigma g\|_{L^p(\R^n)}\le \|g\|_{L^p(\R^n)}.
$$
Since
$$
I_{s,\alpha}=F_\alpha^{-1}I_sF_\alpha,
\qquad
J_{\sigma,\alpha}=F_\alpha^{-1}J_\sigma F_\alpha,
$$
Theorem~\ref{thm:pullback-scales} applied to $I_s$ and $J_\sigma$ gives
$$
\|I_{s,\alpha}f\|_{L^q_\alpha}\le C\|f\|_{L^p_\alpha},
\qquad
\|J_{\sigma,\alpha}f\|_{L^p_\alpha}\le \|f\|_{L^p_\alpha}.
$$
This proves the theorem.
\end{proof}

\section{Fractional parameter descriptors}
The fractional order parameter $\rho$ enters the framework through two components. The first is the chirp factor in $\Ma$, and the second is the rescaling hidden in the selector family. We now isolate these two components as explicit parameter descriptors.

\begin{definition}[\textbf{Fractional structural descriptors}]\label{def:descriptors}
For $\rho\in\R\setminus 2\mathbb Z$, define
$$
\alpha(\rho):=\frac{\pi\rho}{2},
\qquad
\kappa(\rho):=\cot\alpha(\rho),
\qquad
s(\rho):=\abs{\sin\alpha(\rho)}.
$$

We call $\kappa(\rho)$ the \emph{chirp-slope descriptor} and $s(\rho)$ the \emph{scale-dilation descriptor}. We also define the scalar \emph{fractional deviation index}
$$
\mathfrak{D}(\rho):=\abs{1-s(\rho)}+\abs{\kappa(\rho)}.
$$
Finally, for $R>0$, define the effective classical bandwidth
$$
R_{\mathrm{eff}}(\rho;R):=\frac{R}{s(\rho)}.
$$
\end{definition}

\begin{remark}
The point of Definition~\ref{def:descriptors}  regard scalar index $\mathfrak D(\rho)$  as the analogue of a macroscopic regime descriptor.
\end{remark}

\begin{theorem}[\textbf{Parameter-action law for localized selectors}]\label{thm:action-law}
Let $\alpha(\rho)=\pi\rho/2\notin \pi\mathbb Z$, let $\Phi$ be a bounded fuction, and let $R>0$. Then for every $u\in\Tem(\R^n)$,
$$
S^{\Phi}_{R,\alpha(\rho)}u
=\Ma^{-1}T_{\phi_{\rho,R}}\Ma u,
\qquad
\phi_{\rho,R}(\xi):=\Phi\!\left(\frac{s(\rho)\xi}{R}\right).
$$
Hence, the fractional parameter acts only through the pair $(\kappa(\rho),s(\rho))$: $\kappa(\rho)$ changes the chirp geometry, while $s(\rho)$ changes the effective bandwidth from $R$ to $R_{\mathrm{eff}}(\rho;R)=R/s(\rho)$. In particular, on the norm scales retained, the operator bounds are preserved, whereas the geometric placement of evidence across components is not.
\end{theorem}

\begin{proof}
By definition,
$$
S^{\Phi}_{R,\alpha}u=\Fa^{-1}\bigl(\Phi(\cdot/R)\Fa u\bigr),
$$
which is equivalent to a classical localized multiplier family after a fixed rescaling by $\abs{\sin\alpha}$. Writing $m(\eta)=\Phi(\eta/R)$ and using FrFT multiplier identity, we obtain
$$
S^{\Phi}_{R,\alpha}=\Ma^{-1}T_{m_{\alpha}}\Ma,
\qquad m_{\alpha}(\xi)=m((\sin\alpha)\xi)=\Phi\!\left(\frac{\abs{\sin\alpha}\,\xi}{R}\right).
$$
Substituting $\alpha=\alpha(\rho)$ gives the formula. The statement on operator norms follows from the norm-preservation mechanism.
\end{proof}

\section{Theoretical framework}

\begin{proposition}[\textbf{Theorem of strong and weak bounds}]\label{prop:21}
Let $T$ be a linear or sublinear operator on functions over $\R^{n}$ and define $T_{\alpha}=\Ma^{-1}T\Ma$. Then the following hold.
\begin{enumerate}[label=\textup{(\roman*)}]
    \item If $T:L^{p}(\R^{n})\to L^{q}(\R^{n})$ is bounded for some $0<p,q\le\infty$, then
    $$
    \norm{T_{\alpha}f}_{L^{q}}\le \norm{T}_{L^{p}\to L^{q}}\,\norm{f}_{L^{p}},
    \qquad f\in L^{p}(\R^{n}),
    $$
    and in fact $\norm{T_{\alpha}}_{L^{p}\to L^{q}}=\norm{T}_{L^{p}\to L^{q}}$.
    \item If $T:L^{1}(\R^{n})\to L^{1,\infty}(\R^{n})$ is of weak type $(1,1)$, then
    $$
    \norm{T_{\alpha}f}_{L^{1,\infty}}\le \norm{T}_{L^{1}\to L^{1,\infty}}\,\norm{f}_{L^{1}},
    \qquad f\in L^{1}(\R^{n}),
    $$
    and again the operator quasi-norm is preserved.
\end{enumerate}
The same argument works for any rearrangement-invariant function space whose norm depends only on the distribution function of the modulus.
\end{proposition}

\begin{proof}
Let $g=\Ma f$. Since $\Ma$ is multiplication by a unimodular factor, one has
$$
\norm{g}_{L^{p}}=\norm{f}_{L^{p}}
\qquad\text{for every }0<p\le \infty.
$$
Moreover,
$$
T_{\alpha}f=\Ma^{-1}T(\Ma f)=\Ma^{-1}Tg,
$$
so the output differs from $Tg$ only by a unimodular factor. Hence
$$
\abs{T_{\alpha}f(x)}=\abs{Tg(x)}
\quad\text{for almost every }x.
$$
If $T$ is bounded from $L^{p}$ to $L^{q}$, then
$$
\norm{T_{\alpha}f}_{L^{q}}
=\norm{\Ma^{-1}Tg}_{L^{q}}
=\norm{Tg}_{L^{q}}
\le \norm{T}_{L^{p}\to L^{q}}\,\norm{g}_{L^{p}}
=\norm{T}_{L^{p}\to L^{q}}\,\norm{f}_{L^{p}}.
$$
This proves the inequality. Applying the same argument to $T=\Ma T_{\alpha}\Ma^{-1}$ shows the reverse inequality, so the operator norms are equal.

For the weak type statement, fix $\lambda>0$. Because $\abs{T_{\alpha}f}=\abs{Tg}$,
$$
\big|\set{x:\abs{T_{\alpha}f(x)}>\lambda}\big|
=\big|\set{x:\abs{Tg(x)}>\lambda}\big|.
$$
Therefore, we have
$$
\lambda\,\big|\set{x:\abs{T_{\alpha}f(x)}>\lambda}\big|
=\lambda\,\big|\set{x:\abs{Tg(x)}>\lambda}\big|
\le \norm{T}_{L^{1}\to L^{1,\infty}}\,\norm{g}_{L^{1}}
=\norm{T}_{L^{1}\to L^{1,\infty}}\,\norm{f}_{L^{1}}.
$$
This finishes the proof of Proposition $\ref{prop:21}$.
\end{proof}

\begin{proposition}[\textbf{Dyadic conjugation principle}]\label{prop:81}
Let $u\in\mathscr S'(\R^{n})$, set $v=\Ma u$, let $s\in\R$, $1\le p\le\infty$, and $0<q\le\infty$. Then the transported low-frequency and dyadic pieces satisfy
$$
\norm{S_{0,\alpha}u}_{L^{p}}=\norm{S_{0}v}_{L^{p}},
$$
$$
\Big(\sum_{j\ge1}\big(2^{js}\norm{\Delta_{j,\alpha}u}_{L^{p}}\big)^{q}\Big)^{1/q}
=
\Big(\sum_{j\ge1}\big(2^{js}\norm{\Delta_{j}v}_{L^{p}}\big)^{q}\Big)^{1/q},
$$
with the usual supremum interpretation when $q=\infty$, and, whenever $p<\infty$,
$$
\Big\|\Big(\sum_{j\ge1}\big(2^{js}\abs{\Delta_{j,\alpha}u}\big)^{q}\Big)^{1/q}\Big\|_{L^{p}}
=
\Big\|\Big(\sum_{j\ge1}\big(2^{js}\abs{\Delta_{j}v}\big)^{q}\Big)^{1/q}\Big\|_{L^{p}}.
$$
Again, for $q=\infty$ the sum is replaced by an essential supremum.
\end{proposition}

\begin{proof}
By definition, we have
\begin{equation}\label{eq:dyadic-conjugation}
S_{0,\alpha}u=\Ma^{-1}S_{0}(\Ma u)=\Ma^{-1}S_{0}v,
\qquad
\Delta_{j,\alpha}u=\Ma^{-1}\Delta_{j}(\Ma u)=\Ma^{-1}\Delta_{j}v.
\end{equation}
Since $\Ma^{-1}$ multiplies by a unimodular function, it preserves pointwise modulus. Hence
\begin{equation}\label{eq:dyadic-modulus}
\abs{S_{0,\alpha}u}=\abs{S_{0}v},
\qquad
\abs{\Delta_{j,\alpha}u}=\abs{\Delta_{j}v}
\quad \text{a.e.}
\end{equation}
Taking the $L^{p}$ norm in \eqref{eq:dyadic-modulus}, we obtain
\begin{equation}\label{eq:dyadic-lowfreq}
\norm{S_{0,\alpha}u}_{L^{p}}=\norm{S_{0}v}_{L^{p}}.
\end{equation}
Also, by \eqref{eq:dyadic-modulus}, we have
\begin{equation}\label{eq:dyadic-termwise}
2^{js}\norm{\Delta_{j,\alpha}u}_{L^{p}}
=2^{js}\norm{\Delta_{j}v}_{L^{p}},\qquad j\ge1.
\end{equation}
Therefore, we have
\begin{equation}\label{eq:dyadic-besov}
\left(\sum_{j\ge1}\bigl(2^{js}\norm{\Delta_{j,\alpha}u}_{L^{p}}\bigr)^{q}\right)^{1/q}
=
\left(\sum_{j\ge1}\bigl(2^{js}\norm{\Delta_{j}v}_{L^{p}}\bigr)^{q}\right)^{1/q},
\end{equation}
with the usual supremum interpretation when $q=\infty$. Moreover, \eqref{eq:dyadic-modulus} also gives
\begin{equation}\label{eq:dyadic-triebel-pointwise}
\left(\sum_{j\ge1}\bigl(2^{js}\abs{\Delta_{j,\alpha}u}\bigr)^{q}\right)^{1/q}
=
\left(\sum_{j\ge1}\bigl(2^{js}\abs{\Delta_{j}v}\bigr)^{q}\right)^{1/q}
\quad \text{a.e.},
\end{equation}
again with the usual supremum interpretation when $q=\infty$. Taking the $L^{p}$ norm in \eqref{eq:dyadic-triebel-pointwise}, we arrive at
\begin{equation}\label{eq:dyadic-triebel}
\norm{
\left(\sum_{j\ge1}\bigl(2^{js}\abs{\Delta_{j,\alpha}u}\bigr)^{q}\right)^{1/q}
}_{L^{p}}
=
\norm{
\left(\sum_{j\ge1}\bigl(2^{js}\abs{\Delta_{j}v}\bigr)^{q}\right)^{1/q}
}_{L^{p}}.
\end{equation}
Thus, combining \eqref{eq:dyadic-lowfreq} and \eqref{eq:dyadic-besov} with \eqref{eq:dyadic-triebel} yield the desired conclusion.

For the Triebel-Lizorkin-type expression, the pointwise equality of moduli implies
$$
\Big(\sum_{j\ge1}\big(2^{js}\abs{\Delta_{j,\alpha}u(x)}\big)^{q}\Big)^{1/q}
=
\Big(\sum_{j\ge1}\big(2^{js}\abs{\Delta_{j}v(x)}\big)^{q}\Big)^{1/q}
$$
for almost every $x$, again with the usual essential supremum interpretation when $q=\infty$. Taking the $L^{p}$ norm gives the last identity.
\end{proof}

\begin{theorem}[\textbf{FrFT inhomogeneous Sobolev, Besov, and Triebel-Lizorkin geometry}]\label{thm:82}
Let $s\in\R$ and $\alpha\in\R\setminus\pi\mathbb Z$. For $u\in\mathscr S'(\R^{n})$ define
$$
\norm{u}_{B^{s}_{p,q,\alpha}}
:=\norm{S_{0,\alpha}u}_{L^{p}}
+\Big(\sum_{j\ge1}\big(2^{js}\norm{\Delta_{j,\alpha}u}_{L^{p}}\big)^{q}\Big)^{1/q}
$$
and
$$
\norm{u}_{F^{s}_{p,q,\alpha}}
:=\norm{S_{0,\alpha}u}_{L^{p}}
+\Big\|\Big(\sum_{j\ge1}\big(2^{js}\abs{\Delta_{j,\alpha}u}\big)^{q}\Big)^{1/q}\Big\|_{L^{p}},
$$
with the usual supremum interpretation when $q=\infty$. In particular, define the inhomogeneous Sobolev norm by
$$
\norm{u}_{H^{s,p}_{\alpha}}:=\norm{u}_{F^{s}_{p,2,\alpha}}.
$$
If $v=\Ma u$, then, whenever the corresponding classical norms,
$$
\norm{u}_{B^{s}_{p,q,\alpha}}=\norm{v}_{B^{s}_{p,q}},
\qquad
\norm{u}_{F^{s}_{p,q,\alpha}}=\norm{v}_{F^{s}_{p,q}},
\qquad
\norm{u}_{H^{s,p}_{\alpha}}=\norm{v}_{H^{s,p}}.
$$
Thus, the FrFT inhomogeneous spaces are the chirp conjugation versions of the classical ones.
\end{theorem}

\begin{proof}
The proposition above already gives all needed identities. Indeed,
$$
\norm{S_{0,\alpha}u}_{L^{p}}=\norm{S_{0}v}_{L^{p}},
\qquad
2^{js}\norm{\Delta_{j,\alpha}u}_{L^{p}}=2^{js}\norm{\Delta_{j}v}_{L^{p}},
$$
and
$$
\Big\|\Big(\sum_{j\ge1}\big(2^{js}\abs{\Delta_{j,\alpha}u}\big)^{q}\Big)^{1/q}\Big\|_{L^{p}}
=
\Big\|\Big(\sum_{j\ge1}\big(2^{js}\abs{\Delta_{j}v}\big)^{q}\Big)^{1/q}\Big\|_{L^{p}}.
$$

Substituting these equalities into the definitions of the FrFT-side Besov and Triebel--Lizorkin norms yields
$$
\norm{u}_{B^{s}_{p,q,\alpha}}=\norm{v}_{B^{s}_{p,q}}
\quad\text{and}\quad
\norm{u}_{F^{s}_{p,q,\alpha}}=\norm{v}_{F^{s}_{p,q}}.
$$
The Sobolev identity is the special case $q=2$, namely
$$
\norm{u}_{H^{s,p}_{\alpha}}=\norm{u}_{F^{s}_{p,2,\alpha}}=\norm{v}_{F^{s}_{p,2}}=\norm{v}_{H^{s,p}}.
$$
No further argument is required: these norms are not merely equivalent to the classical norms; they coincide with the classical norms under chirp conjugation.
\end{proof}

\begin{theorem}[\textbf{Odd-even banks and vector/Banach-valued invariance}]\label{thm:83}
Let $1\le p\le\infty$, and let $0<r,s\le\infty$. Suppose $\{O_{\nu}\}_{\nu\in\mathcal N}$ is a finite family of classical odd-kernel descriptor operators and $\{E_{\mu}\}_{\mu\in\mathcal M}$ is a finite family of classical even-kernel reconstruction operators such that
$$
\Big\|\Big(\sum_{\nu\in\mathcal N}\abs{O_{\nu}g}^{s}\Big)^{1/s}\Big\|_{L^{p}}
\le C_{\mathcal O}\norm{g}_{L^{p}}
\quad \text{and}\quad
\Big\|\Big(\sum_{\mu\in\mathcal M}\abs{E_{\mu}g}^{r}\Big)^{1/r}\Big\|_{L^{p}}
\le C_{\mathcal E}\norm{g}_{L^{p}},
$$
with the usual supremum interpretation for $s=\infty$ or $r=\infty$. Define the conjugated banks
$$
O_{\nu,\alpha}:=\Ma^{-1}O_{\nu}\Ma
\quad \text{and}\quad
E_{\mu,\alpha}:=\Ma^{-1}E_{\mu}\Ma.
$$
Then, we have
$$
\Big\|\Big(\sum_{\nu\in\mathcal N}\abs{O_{\nu,\alpha}f}^{s}\Big)^{1/s}\Big\|_{L^{p}}
\le C_{\mathcal O}\norm{f}_{L^{p}}
\quad \text{and}\quad
\Big\|\Big(\sum_{\mu\in\mathcal M}\abs{E_{\mu,\alpha}f}^{r}\Big)^{1/r}\Big\|_{L^{p}}
\le C_{\mathcal E}\norm{f}_{L^{p}}.
$$
More generally, if $X$ and $Y$ are complex Banach spaces and $T:L^{p}(\R^{n};X)\to L^{p}(\R^{n};Y)$ is bounded, then the conjugated operator
$T_{\alpha}:=\Ma^{-1}T\Ma$
acts boundedly from $L^{p}(\R^{n};X)$ to $L^{p}(\R^{n};Y)$ with the same operator norm.
\end{theorem}

\begin{proof}
Let $g=\Ma f$. Then for each $\nu\in\mathcal N$,
$$
O_{\nu,\alpha}f=\Ma^{-1}O_{\nu}g
\quad \text{and}\quad
\abs{O_{\nu,\alpha}f}=\abs{O_{\nu}g}.
$$
Therefore, 
$$
\Big(\sum_{\nu\in\mathcal N}\abs{O_{\nu,\alpha}f}^{s}\Big)^{1/s}
=
\Big(\sum_{\nu\in\mathcal N}\abs{O_{\nu}g}^{s}\Big)^{1/s}
\quad\text{pointwise},
$$
and hence
$$
\Big\|\Big(\sum_{\nu\in\mathcal N}\abs{O_{\nu,\alpha}f}^{s}\Big)^{1/s}\Big\|_{L^{p}}
=
\Big\|\Big(\sum_{\nu\in\mathcal N}\abs{O_{\nu}g}^{s}\Big)^{1/s}\Big\|_{L^{p}}
\le C_{\mathcal O}\norm{g}_{L^{p}}
=C_{\mathcal O}\norm{f}_{L^{p}}.
$$
The proof for the even bank is identical.

For the Banach-valued statement, let $F:\R^{n}\to X$ be strongly measurable. Since $\Ma$ acts by scalar multiplication,
$$
\norm{\Ma F(x)}_{X}=\abs{e^{i\phi_{\alpha}(x)}}\norm{F(x)}_{X}=\norm{F(x)}_{X}
\quad\text{for almost every }x.
$$
Thus, $\Ma$ is an isometry on $L^{p}(\R^{n};X)$ and similarly on $L^{p}(\R^{n};Y)$. Consequently,
\begin{align*}
\norm{T_{\alpha}F}_{L^{p}(\R^{n};Y)}
&=\norm{T(\Ma F)}_{L^{p}(\R^{n};Y)}\\
&\le \norm{T}_{L^{p}(X)\to L^{p}(Y)}\norm{\Ma F}_{L^{p}(\R^{n};X)}\\
&=\norm{T}_{L^{p}(X)\to L^{p}(Y)}\norm{F}_{L^{p}(\R^{n};X)}.
\end{align*}
Applying the same argument to $T=\Ma T_{\alpha}\Ma^{-1}$ gives equality of the operator norms.
\end{proof}

\begin{theorem}[\textbf{Twisted multilinear collaborative modules}]\label{thm:84}
Let $m\ge2$ and let exponents $1\le p_{1},\dots,p_{m}\le\infty$ satisfy
$$
\frac{1}{p}=\frac{1}{p_{1}}+\cdots+\frac{1}{p_{m}}
$$
for some $0<p\le\infty$. Define the twisted product
$$
\Pi_{\alpha}^{(m)}(f_{1},\dots,f_{m})
:=\Ma^{-1}\big((\Ma f_{1})\cdots(\Ma f_{m})\big).
$$
Then, the following inequality holds,
$$
\norm{\Pi_{\alpha}^{(m)}(f_{1},\dots,f_{m})}_{L^{p}}
\le \prod_{j=1}^{m}\norm{f_{j}}_{L^{p_{j}}}.
$$

Assume further that $\sigma(\xi_{1},\dots,\xi_{m})$ is a classical Coifman-Meyer symbol, so the associated multilinear Fourier multiplier $T_{\sigma}$ satisfies
$$
\norm{T_{\sigma}(g_{1},\dots,g_{m})}_{L^{p}}
\le C_{\sigma,\bm p}\prod_{j=1}^{m}\norm{g_{j}}_{L^{p_{j}}}
$$
throughout the usual boundedness range. Define its FrFT conjugate by
$$
T_{\sigma,\alpha}(f_{1},\dots,f_{m})
:=\Ma^{-1}T_{\sigma}(\Ma f_{1},\dots,\Ma f_{m}).
$$
Then, the following inequality also holds,
$$
\norm{T_{\sigma,\alpha}(f_{1},\dots,f_{m})}_{L^{p}}
\le C_{\sigma,\bm p}\prod_{j=1}^{m}\norm{f_{j}}_{L^{p_{j}}}.
$$
Hence every classical multilinear Coifman-Meyer estimate can develops to FrFT version.
\end{theorem}

\begin{proof}
For the twisted product, unimodularity gives the pointwise identity
$$
\abs{\Pi_{\alpha}^{(m)}(f_{1},\dots,f_{m})}
=\prod_{j=1}^{m}\abs{\Ma f_{j}}
=\prod_{j=1}^{m}\abs{f_{j}}.
$$
Applying multilinear Holder's inequality immediately yields
$$
\norm{\Pi_{\alpha}^{(m)}(f_{1},\dots,f_{m})}_{L^{p}}
\le \prod_{j=1}^{m}\norm{f_{j}}_{L^{p_{j}}}.
$$
For the multiplier module, set $g_{j}=\Ma f_{j}$. Then
$$
T_{\sigma,\alpha}(f_{1},\dots,f_{m})=\Ma^{-1}T_{\sigma}(g_{1},\dots,g_{m}),
$$
so by the isometric property of $\Ma^{-1}$,
$$
\norm{T_{\sigma,\alpha}(f_{1},\dots,f_{m})}_{L^{p}}
=\norm{T_{\sigma}(g_{1},\dots,g_{m})}_{L^{p}}.
$$
The classical Coifman-Meyer theorem therefore gives
$$
\norm{T_{\sigma,\alpha}(f_{1},\dots,f_{m})}_{L^{p}}
\le C_{\sigma,\bm p}\prod_{j=1}^{m}\norm{g_{j}}_{L^{p_{j}}}
=C_{\sigma,\bm p}\prod_{j=1}^{m}\norm{f_{j}}_{L^{p_{j}}}.
$$
This is the desired FrFT bound.
\end{proof}

\begin{remark}
If one rewrites $T_{\sigma,\alpha}$ in FrFT frequency variables, the resulting symbol is a rescaled version of the classical symbol. Thus the result may also be interpreted as showing that the Coifman--Meyer class is invariant under FrFT conjugation.
\end{remark}

\section{The oscillation package on $\BMO_{\alpha}$}

We recall that the classical BMO seminorm can be written as
$$
\norm{g}_{\BMO}
:=\sup_Q \frac1{|Q|}\int_Q \abs{g(x)-\Avg_Q(g)}\,dx,
$$
where the supremum is over cubes $Q\subset\R^n$.

\begin{theorem}[\textbf{FrFT John-Nirenberg theorem}]\label{thm:JN}
Let $b\in \BMOa(\R^n)$. Then there exist constants $C_1,C_2>0$, depending only on the dimension, such that for every cube $Q\subset\R^n$ and every $\lambda>0$,
$$
\bigl|\{x\in Q:\abs{\Ma b(x)-\Avg_Q(\Ma b)}>\lambda\}\bigr|
\le C_1|Q|\exp\!\paren{-\frac{C_2\lambda}{\norm{b}_{\BMOa}}}.
$$
Equivalently, the oscillation of $\Ma b$ around its mean has exponential decay uniformly over cubes.
\end{theorem}

\begin{proof}
This is Proposition \ref{prop:general} applied to the classical John-Nirenberg theorem for $g=\Ma b$. Indeed, the quantity being estimated is the distribution function of the oscillation $\abs{g-\Avg_Q(g)}$, while the controlling norm is $\norm{g}_{\BMO}=\norm{b}_{\BMOa}$. After replacing $g$ by $\Ma b$, the inequality is the displayed FrFT statement.
\end{proof}

\begin{corollary}[\textbf{$L^r$ oscillation equivalence on $\BMO_{\alpha}$}]\label{cor:LrBMOa}
For every $1<r<\infty$ there exist constants $c_r,C_r>0$, depending only on $n$ and $r$, such that for all $b\in\BMOa(\R^n)$,
$$
 c_r\norm{b}_{\BMOa}
 \le
 \sup_Q\paren{\frac1{|Q|}\int_Q \abs{\Ma b(x)-\Avg_Q(\Ma b)}^r\,dx}^{1/r}
 \le C_r\norm{b}_{\BMOa}.
$$
Hence, the classical mean oscillation norm and every $L^r$ oscillation norm are equivalent on $\BMO_{\alpha}$.
\end{corollary}

\begin{proof}
Set again $g=\Ma b$. The lower bound is elementary: for each cube $Q$, H\"{o}lder's inequality gives
$$
\frac1{|Q|}\int_Q \abs{g-g_Q}\,dx
\le \paren{\frac1{|Q|}\int_Q \abs{g-g_Q}^r\,dx}^{1/r},
$$
where $g_Q:=\Avg_Q(g)$.  Taking the supremum over $Q$ yields
$$
\norm{g}_{\BMO}
\le \sup_Q\paren{\frac1{|Q|}\int_Q \abs{g-g_Q}^r\,dx}^{1/r}.
$$
Since $\norm{g}_{\BMO}=\norm{b}_{\BMOa}$, this proves the left-hand inequality with $c_r=1$.

For the upper bound, fix a cube $Q$ and write $u(x):=\abs{g(x)-g_Q}$. By the layer-cake representation,
$$
\frac1{|Q|}\int_Q u(x)^r\,dx
=\frac{r}{|Q|}\int_0^{\infty}\lambda^{r-1}\bigl|\{x\in Q:u(x)>\lambda\}\bigr|\,d\lambda.
$$

Applying Theorem \ref{thm:JN} to $g$ gives
$$
\bigl|\{x\in Q:u(x)>\lambda\}\bigr|
\le C_1|Q|\exp\!\paren{-\frac{C_2\lambda}{\norm{g}_{\BMO}}}.
$$

Substituting this into the layer-cake formula, we obtain
$$
\frac1{|Q|}\int_Q u(x)^r\,dx
\le C_1 r\int_0^{\infty}\lambda^{r-1}
\exp\!\paren{-\frac{C_2\lambda}{\norm{g}_{\BMO}}}
\,d\lambda.
$$

After the change of variables $\lambda=\norm{g}_{\BMO}t/C_2$, this becomes
$$
\frac1{|Q|}\int_Q u(x)^r\,dx
\le C_1 r C_2^{-r}\Gamma(r)\,\norm{g}_{\BMO}^r.
$$

Taking the $r$-th root and then the supremum over cubes gives
$$
\sup_Q\paren{\frac1{|Q|}\int_Q \abs{g-g_Q}^r\,dx}^{1/r}
\le C_r\norm{g}_{\BMO}=C_r\norm{b}_{\BMOa}.
$$

This is the required upper bound.
\end{proof}

\section{FrFT BMO-Carleson theorem}

Let $T(Q):=Q\times (0,\ell(Q)]$ be the tent above a cube $Q$, where $\ell(Q)$ denotes its sidelength. We assume that $\Psi\in\calS(\R^n)$ satisfies
$$
\int_{\R^n}\Psi(x)\,dx=0
$$
and the usual classical nondegeneracy condition ensuring the converse BMO-Carleson theorem, for instance
$$
\inf_{\xi\neq 0}\sup_{t>0}\abs{\widehat{\Psi}(t\xi)}>0.
$$

\begin{theorem}[\textbf{FrFT BMO-Carleson theorem}]\label{thm:BMOCarleson}
For $b\in\calS'(\R^n)$, the following are equivalent:
\begin{enumerate}[label=\textup{(\roman*)}]
\item $b\in \BMOa(\R^n)$;
\item the measure $\mu_b^{\alpha}$ defined by
$$
 d\mu_b^{\alpha}(x,t)=\abs{\Psi_t^{\alpha}b(x)}^2\,dx\,\frac{dt}{t}
$$
is a Carleson measure on $\R^{n+1}_+$, that is,
$$
\sup_Q \frac{\mu_b^{\alpha}(T(Q))}{|Q|}<\infty.
$$
\end{enumerate}
Moreover, there exist constants $c,C>0$, depending only on $n$ and $\Psi$, such that
$$
 c\norm{b}_{\BMOa}
 \le
 \sup_Q\paren{\frac1{|Q|}\int_{T(Q)}\abs{\Psi_t^{\alpha}b(x)}^2\,dx\,\frac{dt}{t}}^{1/2}
 \le
 C\norm{b}_{\BMOa}.
$$
\end{theorem}

\begin{proof}
Let $g:=\Ma b$. By Lemma \ref{lem:isom}, we have that for all $(x,t)\in \R^{n+1}_+$
$$
\abs{\Psi_t^{\alpha}b(x)}=\abs{\Psi_t*g(x)}.
$$
It is immediately obtain that the upper-half-space measure $\mu_b^{\alpha}$ is the classical square-function measure $\mu_g$ associated with $g$. Since also $\norm{g}_{\BMO}=\norm{b}_{\BMOa}$, the present statement is precisely the pullback of the classical BMO-Carleson theorem under $g=\Ma b$. Both implications and the displayed norm equivalence therefore follow immediately from the classical theorem.
\end{proof}

\section{Sharp maximal estimates and endpoint bounds under chirp conjugation}

\begin{proposition}[\textbf{Sharp-maximal characterization of $\BMO_{\alpha}$}]\label{prop:sharp}
Let $f$ be locally integrable function. Then, we have
$$
 f\in\BMOa(\R^n)
 \quad\Longleftrightarrow\quad
 \Msharp_{\alpha}f\in L^{\infty}(\R^n).
$$
Moreover, there exist constants $c,C>0$ depending only on the dimension such that
$$
 c\norm{f}_{\BMOa}
 \le
 \norm{\Msharp_{\alpha}f}_{L^{\infty}}
 \le
 C\norm{f}_{\BMOa}.
$$
\end{proposition}

\begin{proof}
This is again an immediate application of Proposition \ref{prop:general}, now to the classical Fefferman-Stein characterization
$$
 g\in\BMO(\R^n)
 \quad\Longleftrightarrow\quad
 \Msharp g\in L^{\infty}(\R^n),
\qquad
 \norm{g}_{\BMO}\asymp \norm{\Msharp g}_{L^{\infty}}.
$$
Indeed, with $g=\Ma f$ one has $\Msharp_{\alpha}f=\Msharp g$ and $\norm{f}_{\BMOa}=\norm{g}_{\BMO}$, so the FrFT statement is the pullback of the classical one.
\end{proof}

\begin{theorem}[\textbf{Endpoint bounds and interpolation consequences}]\label{thm:endpoint}
Let $T$ be a classical linear operator on $\R^n$ and define its FrFT operator by
$$
 T_{\alpha}:=\Mai T\Ma.
$$
Assume that for some exponents $1\le p_0,q_0<\infty$,
$$
\norm{Tg}_{L^{q_0}}\le A\norm{g}_{L^{p_0}}
\qquad\text{for all }g\in L^{p_0}(\R^n),
$$
and
$$
\norm{Tg}_{\BMO}\le B\norm{g}_{L^{\infty}}
\qquad\text{for all }g\in L^{\infty}(\R^n).
$$
Then
$$
\norm{T_{\alpha}f}_{L^{q_0}}\le A\norm{f}_{L^{p_0}},
\qquad
\norm{T_{\alpha}f}_{\BMOa}\le B\norm{f}_{L^{\infty}}.
$$
In addition, every classical estimate deduced from these endpoint bounds by an interpolation argument transfers  to $T_{\alpha}$ with the same constants.
\end{theorem}

\begin{proof}
Set $g:=\Ma f$. Then $T_{\alpha}f=\Mai Tg$, so the unimodularity of $\Mai$ and the definition of $\BMOa$ give
$$
\norm{T_{\alpha}f}_{L^{q_0}}=\norm{Tg}_{L^{q_0}},
\qquad
\norm{T_{\alpha}f}_{\BMOa}=\norm{Tg}_{\BMO},
$$
while $\norm{g}_{L^{p_0}}=\norm{f}_{L^{p_0}}$ and $\norm{g}_{L^{\infty}}=\norm{f}_{L^{\infty}}$. The two endpoint bounds therefore follow immediately from the classical assumptions on $T$. Any interpolation consequence is obtained in the same way: apply the classical interpolation theorem to $g$, then transport the resulting estimate back through the isometric identifications $X_{\alpha}=\Mai X$ and $Y_{\alpha}=\Mai Y$.
\end{proof}

\section{Equivalent stability and regime invariance}

For $1<r<\infty$, define the three FrFT-domain stability scores
\begin{align*}
\Omega_{\alpha,r}(b)
&:=\sup_Q\paren{\frac1{|Q|}\int_Q \abs{\Ma b(x)-\Avg_Q(\Ma b)}^r\,dx}^{1/r},
\\
\calC_{\alpha}(b)
&:=\sup_Q\paren{\frac1{|Q|}\int_{T(Q)}\abs{\Psi_t^{\alpha}b(x)}^2\,dx\,\frac{dt}{t}}^{1/2},
\\
\calM_{\alpha}(b)
&:=\norm{\Msharp_{\alpha} b}_{L^{\infty}}.
\end{align*}

\begin{theorem}[\textbf{Equivalent oscillation, Carleson, and sharp-maximal function in $\BMOa $}]\label{thm:equivalent}
For every $b\in\BMOa(\R^n)$ and every $1<r<\infty$,
$$
\Omega_{\alpha,r}(b)\asymp \calC_{\alpha}(b)\asymp \calM_{\alpha}(b)\asymp \norm{b}_{\BMOa},
$$
where the equivalence constants depend only on $n$, $r$, and the choice of $\Psi$.
Consequently, Consequently, for every family $\{b_{\alpha,e}\}_{e\in E}$ indexed,
$$
\sup_{e\in E}\Omega_{\alpha,r}(b_{\alpha,e})<\infty
\iff
\sup_{e\in E}\calC_{\alpha}(b_{\alpha,e})<\infty
\iff
\sup_{e\in E}\calM_{\alpha}(b_{\alpha,e})<\infty.
$$
\end{theorem}

\begin{proof} This proof is trivial. Indeed,
Corollary \ref{cor:LrBMOa} gives  $
\Omega_{\alpha,r}(b)\asymp \norm{b}_{\BMOa}.$
Theorem \ref{thm:BMOCarleson} gives
$$
\calC_{\alpha}(b)\asymp \norm{b}_{\BMOa}.
$$
Proposition \ref{prop:sharp} gives
$$
\calM_{\alpha}(b)\asymp \norm{b}_{\BMOa}.
$$
Combining the three equivalences yields the first assertion. The statement for families follows by taking suprema over $e\in E$.
\end{proof}

\begin{corollary}[\textbf{Operational regime invariance under score change}]\label{cor:regime}
Let $S_1,S_2\in\{\Omega_{\alpha,r},\calC_{\alpha},\calM_{\alpha}\}$. Then there exist constants $c,C>0$ such that
$$
 cS_1(b)\le S_2(b)\le C S_1(b)
\qquad\text{for all }b\in\BMOa(\R^n).
$$
Hence a stable, warning, or failure partition defined by thresholds for one score is unchanged up to a deterministic recalibration of thresholds for any other score.
\end{corollary}

\begin{proof}
Theorem \ref{thm:equivalent} shows that $S_1$ and $S_2$ are equivalent norms on $\BMOa$ modulo chirped constants. Thus the level sets of $S_1$ and $S_2$ differ only by the constants $c$ and $C$. For instance,
$$
 S_1(b)\le \tau
 \implies
 S_2(b)\le C\tau,
\qquad
 S_2(b)>\tau
 \implies
 S_1(b)>\tau/C.
$$
This is the claimed regime invariance after recalibration.
\end{proof}

\section{FrFT Hardy square functions and FrFT atomic decomposition}

Fix $0<p\le 1$ and choose a standard homogeneous Littlewood-Paley operator $\{\Delta_j\}_{j\in\Z}$. Recall that classically, for $g\in\calS'(\R^n)$ one has
$$
S(g):=\paren{\sum_{j\in\Z}\abs{\Delta_j g}^2}^{1/2}.
$$

\begin{definition}[\textbf{FrFT atoms}]
Let $0<p\le 1$ and $1<q\le\infty$. A function $A_{\alpha}$ is called an FrFT $L^q$ atom if there exists a cube $Q\subset\R^n$ such that
\begin{enumerate}[label=\textup{(\roman*)}]
\item $\supp(A_{\alpha})\subset Q$,
\item $\norm{A_{\alpha}}_{L^q}\le |Q|^{1/q-1/p}$,
\item for every multi-index $\gamma$ with $|\gamma|\le \lfloor n(1/p-1)\rfloor$,
$$
\int_{\R^n} x^{\gamma} e^{i\pi |x|^2\cot\alpha}A_{\alpha}(x)\,dx=0.
$$
\end{enumerate}
\end{definition}

\begin{proposition}[\textbf{Hardy identities under chirp conjugation}]\label{prop:HardyTransport}
Let $0<p\le 1$, $1<q\le \infty$, and write $g:=\Ma f$.
\begin{enumerate}[label=\textup{(\roman*)}]
\item One has $\norm{f}_{\Hpa}=\norm{g}_{\Hp}$ and $S_{\alpha}(f)=S(g)$ pointwise.
\item A function $Q_{\alpha}$ belongs to $\Palpha$ if and only if $\Ma Q_{\alpha}\in\calP$.
\item A function $A_{\alpha}$ is an FrFT $L^q$ atom if and only if $a:=\Ma A_{\alpha}$ is a classical $L^q$ atom.
\end{enumerate}
Hence every classical theorem formulated only in terms of $\Hp$ norms, Littlewood-Paley square functions, polynomial ambiguity, or atomic decompositions  to the FrFT side under conjugation by $\Ma$.
\end{proposition}

\begin{proof}
Part \textup{(i)} is immediate from the definitions and Lemma \ref{lem:isom}, since $\Delta_{j,\alpha}f=\Mai\Delta_j g$ and therefore $\abs{\Delta_{j,\alpha}f}=\abs{\Delta_j g}$. Part \textup{(ii)} is the definition $\Palpha=\Mai\calP$. For part \textup{(iii)}, unimodularity preserves support and $L^q$ norms, while
$$
\int_{\R^n}x^{\gamma}e^{i\pi|x|^2\cot\alpha}A_{\alpha}(x)\,dx
=\int_{\R^n}x^{\gamma}a(x)\,dx.
$$
Thus, the chirped moment conditions for $A_{\alpha}$ are the classical vanishing moments for $a$.
\end{proof}

\begin{theorem}[\textbf{FrFT Hardy square-function characterization}]\label{thm:HardySquare}
Let $0<p\le 1$.
\begin{enumerate}[label=\textup{(\roman*)}]
\item If $f\in\Hpa(\R^n)$, then
$$
\norm{f}_{\Hpa}\asymp \norm{S_{\alpha}(f)}_{L^p}
=\norm{\paren{\sum_{j\in\Z}\abs{\Delta_{j,\alpha}f}^2}^{1/2}}_{L^p}.
$$

\item Conversely, if $S_{\alpha}(f)\in L^p(\R^n)$, then there exists a unique $Q_{\alpha}\in \Palpha$ such that
$$
 f-Q_{\alpha}\in\Hpa(\R^n)\quad\text{and}\quad \norm{f-Q_{\alpha}}_{\Hpa}\asymp \norm{S_{\alpha}(f)}_{L^p}.
$$
\end{enumerate}
The equivalence constants depend only on $n$, $p$, and the choice of Littlewood-Paley partition.
\end{theorem}

\begin{proof}
By Proposition \ref{prop:HardyTransport}, the present statement is the pullback of the classical Hardy square-function theorem under $g=\Ma f$: part \textup{(i)} gives the identities $\norm{f}_{\Hpa}=\norm{g}_{\Hp}$ and $S_{\alpha}(f)=S(g)$, while part \textup{(ii)} identifies the FrFT polynomial ambiguity $Q_{\alpha}\in\Palpha$ with the polynomial ambiguity $P=\Ma Q_{\alpha}\in\calP$. Therefore the classical theorem for $g$ is equivalent, after conjugation by $\Ma$, to the displayed FrFT statement, including the uniqueness of $Q_{\alpha}$.
\end{proof}

\begin{theorem}[\textbf{FrFT atomic decomposition}]\label{thm:atomic}
Let $0<p\le 1$ and $1<q\le\infty$. A tempered distribution $f$ belongs to $\Hpa(\R^n)$ if and only if there exist FrFT $L^q$ atoms $\{A_{\nu,\alpha}\}_{\nu}$ and coefficients $\{\lambda_{\nu}\}_{\nu}$ such that
$$
 f=\sum_{\nu}\lambda_{\nu}A_{\nu,\alpha}
 \qquad\text{in }\calS'(\R^n)
$$
and
$
\displaystyle\sum_{\nu}\abs{\lambda_{\nu}}^p<\infty.
$
Moreover,
$$
\norm{f}_{\Hpa}
\asymp
\inf\paren{\sum_{\nu}\abs{\lambda_{\nu}}^p}^{1/p},
$$
where the infimum runs over all such atomic representations.
\end{theorem}

\begin{proof}
By Proposition \ref{prop:HardyTransport}\textup{(iii)}, a function $A_{\nu,\alpha}$ is an FrFT $L^q$ atom if and only if $a_{\nu}:=\Ma A_{\nu,\alpha}$ is a classical $L^q$ atom. Hence $f\in\Hpa$ if and only if $g:=\Ma f\in\Hp$, and an atomic decomposition of $g$ by classical atoms is equivalent, after applying $\Mai$ termwise, to an atomic decomposition of $f$ by FrFT atoms with the same coefficients. The classical atomic decomposition theorem for $\Hp$ therefore transports  to the present FrFT setting and yields both the representation
$$
 f=\sum_{\nu}\lambda_{\nu}A_{\nu,\alpha}
 \qquad\text{in }\calS'(\R^n)
$$
and the quasi-norm equivalence
$$
\norm{f}_{\Hpa}
\asymp
\inf\paren{\sum_{\nu}\abs{\lambda_{\nu}}^p}^{1/p}.
$$
This proof is completed.
\end{proof}

\begin{corollary}\label{cor:bridge}
For $0<p\le 1$, define the FrFT Triebel-Lizorkin space
$$
\dot F^{0,2}_{p,\alpha}(\R^n)
:=\set{f\in\calS'(\R^n)/\Palpha:S_{\alpha}(f)\in L^p(\R^n)}
$$
with quasi-norm $\norm{f}_{\dot F^{0,2}_{p,\alpha}}:=\norm{S_{\alpha}(f)}_{L^p}$. Then
$$
\dot F^{0,2}_{p,\alpha}(\R^n)=\Hpa(\R^n)
$$
with equivalent quasi-norms.
\end{corollary}

\begin{proof}
Theorem \ref{thm:HardySquare} shows precisely that $f$ belongs to $\Hpa$ if and only if its FrFT square function belongs to $L^p$, modulo the natural ambiguity by elements of $\Palpha$. This is the definition of $\dot F^{0,2}_{p,\alpha}$, and the corresponding quasi-norms are equivalent.
\end{proof}

\section{The validity boundary associated with twisted Kato--Ponce sharpness}

We now formalize the boundary statement used in the anomaly-weighting module. Consider
$$
G_{\omega}(u):=\omega\staralpha u.
$$

Assume the following classical sharpness fact is known: if the estimate
\begin{equation}\label{eq:classicalKP}
\norm{D^s(F*G)}_{L^r}
\le C\paren{\norm{D^sF}_{L^{p_1}}\norm{G}_{L^{p_2}}+\norm{F}_{L^{q_1}}\norm{D^sG}_{L^{q_2}}}
\end{equation}
holds uniformly for all Schwartz functions $F,G$, where
$$
\frac1r=\frac1{p_1}+\frac1{p_2}=\frac1{q_1}+\frac1{q_2}
\quad\text{and}\quad \frac12<r<\infty.
$$
Then, it is necessarily to ensure that
$$
 s>\max\paren{0,\frac nr-n}
 \qquad\text{or}\qquad
 s\in 2\Z_{+}.
$$
Equivalently, the estimate fails in the forbidden region
$$
 s<0,
 \qquad\text{or}\qquad
 0\le s\le \max\paren{\frac nr-n,0}
 \text{ with }
 s\notin 2\Z_{+}\cup\{0\}.
$$

\begin{theorem}[\textbf{FrFT necessary structural condition}]\label{thm:KPsharp}
Assume that for all $\omega$, $u$ one has the uniform FrFT estimate
\begin{equation}\label{eq:frftKP}
\norm{D^s_{\alpha}G_{\omega}(u)}_{L^r}
\le C\paren{\norm{D^s_{\alpha}\omega}_{L^{p_1}}\norm{u}_{L^{p_2}}+\norm{\omega}_{L^{q_1}}\norm{D^s_{\alpha}u}_{L^{q_2}}}.
\end{equation}
Then the same necessary restriction on $s$ must hold:
$$
 s>\max\paren{0,\frac nr-n}
 \qquad\text{or}\qquad
 s\in 2\Z_{+}.
$$
Consequently, \eqref{eq:frftKP} cannot hold uniformly in the forbidden region
$$
 s<0,
 \qquad\text{or}\qquad
 0\le s\le \max\paren{\frac nr-n,0}
 \text{ with }
 s\notin 2\Z_{+}\cup\{0\}.
$$
\end{theorem}

\begin{proof}
Assume that \eqref{eq:frftKP} holds uniformly. For arbitrary Schwartz functions $F,G$, set
$$
\omega:=\Mai F,
\qquad
u:=\Mai G.
$$
Then $\Ma\omega=F$ and $\Ma u=G$, so by definition of the FrFT convolution and derivative,
$$
G_{\omega}(u)=\omega\staralpha u=\Mai(F*G),
\qquad
D^s_{\alpha}G_{\omega}(u)=\Mai D^s(F*G).
$$
Taking $L^p$ norms and using the unimodularity of $\Mai$, we obtain
$$
\norm{D^s_{\alpha}G_{\omega}(u)}_{L^r}=\norm{D^s(F*G)}_{L^r},
\qquad
\norm{D^s_{\alpha}\omega}_{L^{p_1}}=\norm{D^sF}_{L^{p_1}},
\qquad
\norm{u}_{L^{p_2}}=\norm{G}_{L^{p_2}},
$$
with analogous identities for the $L^{q_1}$ and $L^{q_2}$ terms. Substituting these into \eqref{eq:frftKP} shows that the classical estimate \eqref{eq:classicalKP} would then hold for all Schwartz $F,G$ with the same constant $C$. The assumed classical sharpness theorem rules this out in the forbidden region, so the same necessary condition on $s$ must hold in the FrFT setting.
\end{proof}

\section{Fractional order shifting and limit laws}\label{sec:limits}

\begin{theorem}[\textbf{Fractional order shifting in Lipschitz spaces}]\label{thm:ordershift}
Let $\gamma>0$ and $0<\sigma<\gamma$.
\begin{enumerate}[label=\textnormal{(\roman*)}]
\item On the homogeneous scale,
$$
\norm{f}_{\dot\Lambda^{\gamma}}\asymp \norm{|D|^{\sigma}f}_{\dot\Lambda^{\gamma-\sigma}}.
$$
Equivalently, if $g=|D|^{\sigma}f$, then
$$
\norm{|D|^{-\sigma}g}_{\dot\Lambda^{\gamma}}\asymp \norm{g}_{\dot\Lambda^{\gamma-\sigma}}.
$$
\item On the inhomogeneous scale,
$$
\norm{f}_{\Lambda^{\gamma}}\asymp \norm{\langle D\rangle^{\sigma}f}_{\Lambda^{\gamma-\sigma}}.
$$
Equivalently, if $g=\langle D\rangle^{\sigma}f$, then
$$
\norm{\langle D\rangle^{-\sigma}g}_{\Lambda^{\gamma}}\asymp \norm{g}_{\Lambda^{\gamma-\sigma}}.
$$
\end{enumerate}
\end{theorem}

\begin{proof}
Use the classical Littlewood--Paley characterization of homogeneous and inhomogeneous Lipschitz norms. On each dyadic annulus, $|\xi|^{\sigma}$ and $(1+|\xi|^{2})^{\sigma/2}$ differ from the factors $2^{j\sigma}$ by smooth bounded multipliers with kernels of uniformly bounded $L^{1}$ norm. This gives the two-sided estimates on all dyadic pieces, and hence the norm equivalences.
\end{proof}

\begin{proposition}[\textbf{Bandwidth law for localized selectors}]\label{prop:bandlaw}
Let
$$
S^{\Phi}_{R,\alpha}f:=\Fa^{-1}(\Phi(\cdot/R)\,\Fa f).
$$
Then
$$
S^{\Phi}_{R,\alpha}=\Mai T_{\varphi_{\alpha,R}}\Ma,
\qquad
\varphi_{\alpha,R}(\xi)=\Phi\Big(\frac{\sin\alpha}{R}\,\xi\Big).
$$
If $\Phi$ is even or radial, then
$$
\varphi_{\alpha,R}(\xi)=\Phi\Big(\frac{\sa}{R}\,\xi\Big).
$$
If $\Phi$ is supported in $\{1/2\le |\xi|\le 2\}$, then the effective physical band is
$$
\frac{R}{2\sa}\le |\xi|\le \frac{2R}{\sa}.
$$
\end{proposition}

\begin{proof}
Apply Proposition~\ref{prop:frft-mult} to the symbol $m(\xi)=\Phi(\xi/R)$. The support statement is immediate from the rescaled symbol.
\end{proof}

\begin{theorem}[\textbf{Classical-limit law}]\label{thm:classicallimit}
Let $\alpha_{\ell}\to \pi/2+k\pi$ for some $k\in\Z$. Then we have
$$
\sa\to 1,
\qquad
\ka\to 0,
\qquad
D(\alpha_{\ell})\to 0.
$$
Moreover:
\begin{enumerate}[label=\textnormal{(\roman*)}]
\item for every $f\in\calS(\R^{n})$ and every fixed $j\in\Z$,
$$
\Delta_{j,\alpha_{\ell}}f\to \Delta_{j}f
\qquad\text{in }L^{p}(\R^{n}),\quad 1\le p<\infty;
$$
\item for every $f\in\calS(\R^{n})$,
$$
\mathcal S_{\alpha_{\ell}}(f)\to \mathcal S(f)
\qquad\text{in }L^{p}(\R^{n}),\quad 1<p<\infty;
$$
\item if $m\in L^{\infty}(\R^{n})$ is continuous and $T_{m}$ is bounded on $L^{p}(\R^{n})$, then for every $f\in\calS(\R^{n})$,
$$
T_{m,\alpha_{\ell}}f\to T_{m}f
\qquad\text{in }L^{p}(\R^{n});
$$
\item the same convergences hold in $\calS'(\R^{n})$.
\end{enumerate}
\end{theorem}

\begin{proof}
Since $\alpha_{\ell}\to\pi/2+k\pi$, we have $\cot\alpha_{\ell}\to0$ and $|\sin\alpha_{\ell}|\to1$. Hence the chirp factor in $\Ma$ converges pointwise to $1$, and dominated convergence gives $\Ma f\to f$ in every $L^{p}$ for Schwartz $f$. For fixed $j$,
$$
\Delta_{j,\alpha_{\ell}}f-\Delta_{j}f
=\Mai\Delta_{j}(\Ma f-f)+(\Mai-I)\Delta_{j}f,
$$
and both terms tend to zero in $L^{p}$. The square-function convergence follows by dominated convergence in the $\ell^{2}$ variable. The multiplier convergence follows from Proposition~\ref{prop:frft-mult} because $m((\sin\alpha_{\ell})\xi)\to m(\xi)$ pointwise and is uniformly bounded. Distributional convergence is obtained by duality.
\end{proof}

\begin{theorem}[\textbf{Singular-boundary law}]\label{thm:singular}
Let $\alpha_{\ell}\to k\pi$ for some $k\in\Z$. Then
$$
\sa\to 0,
\qquad
|\ka|\to \infty,
\qquad
D(\alpha_{\ell})\to \infty.
$$
Assume that $\Delta_{j}$ is a classical Littlewood--Paley block associated with a nonzero annular cutoff. Then, the following three statements are hold.
\begin{enumerate}[label=\textnormal{(\roman*)}]
\item for every fixed $\xi\ne0$, the rescaled symbol of $\Delta_{j,\alpha_{\ell}}$ tends to $0$ at $\xi$;
\item the effective passband of $\Delta_{j,\alpha_{\ell}}$ drifts to frequencies of size $2^{j}/\sa\to\infty$;
\item for every bounded continuous function $\Phi$ and every $f\in L^{2}(\R^{n})$,
$$
S^{\Phi}_{R,\alpha_{\ell}}f\to \Phi(0)f
\qquad\text{in }L^{2}(\R^{n}).
$$
In particular, if $\Phi(0)=0$, the selector collapses to $0$, whereas if $\Phi(0)=1$, it converges to the identity.
\end{enumerate}
\end{theorem}

\begin{proof}
The first two claims follow from Proposition~\ref{prop:bandlaw}. For the third claim, Proposition~\ref{prop:bandlaw} gives
$$
S^{\Phi}_{R,\alpha_{\ell}}f=\Mai T_{\Phi(\sa\cdot/R)}\Ma f
$$
when $\Phi$ is even or radial, and the same formula with $\sin\alpha_{\ell}$ otherwise. Since the multiplier symbol converges pointwise to the constant $\Phi(0)$ and remains uniformly bounded, Plancherel's theorem yields the claimed $L^{2}$ convergence.
\end{proof}

\begin{corollary}[\textbf{Parameter regimes}]\label{cor:regimes}
Let $0<\delta_{1}<\delta_{2}$. The fractional parameter splits naturally into three regimes:
\begin{enumerate}[label=\textnormal{(\roman*)}]
\item the classical regime $D(\alpha)\le \delta_{1}$;
\item the effective fractional regime $\delta_{1}<D(\alpha)<\delta_{2}$ with $\sa$ bounded away from $0$;
\item the warning regime $D(\alpha)\ge\delta_{2}$ or $\sa$ below the prescribed tolerance.
\end{enumerate}
\end{corollary}

\begin{proof}
This is merely a convenient restatement of Theorems~\ref{thm:classicallimit} and~\ref{thm:singular}.
\end{proof}

\end{document}